\documentclass[12pt]{amsart}

\usepackage[lite]{amsrefs}

\usepackage{amssymb}

\usepackage[all,cmtip]{xy}

\usepackage{geometry}
\usepackage{extarrows}
\usepackage{stmaryrd}
\usepackage{stackrel}
\usepackage{color}

\numberwithin{equation}{section}

\theoremstyle{plain} 
\newtheorem{theorem}{Theorem}[section]

\newtheorem{lemma}[theorem]{Lemma}
\newtheorem{prop}[theorem]{Proposition}

\theoremstyle{definition}
\newtheorem{defn}[theorem]{Definition}

\theoremstyle{remark}
\newtheorem{rem}[theorem]{Remark}

\begin{document}

\title[Pluricomplex Green function of form type Monge-Amp\`{e}re equations]{The pluricomplex Green function of the Monge-Amp\`{e}re equation for $(n-1)$-plurisubharmonic functions and form type $k$-Hessian equations}

\author{Shuimu Li} \email{lsm000524@sjtu.edu.cn}
\thanks{The author was partially supported by Shanghai Jiao Tong University Scientific and Technological Innovation Funds, NSFC-12031012, NSFC-11831003 and the Institute of Modern Analysis-A Frontier Research Center of Shanghai.}


 \curraddr{School of Mathematics, Shanghai Jiaotong University, Shanghai}




\maketitle

\begin{abstract}
In this paper, we introduce the pluricomplex Green function of the Monge-Amp\`{e}re equation for $(n-1)$-plurisubharmonic functions by solving the Dirichlet problem for the form type Monge-Amp\`{e}re and Hessian equations on a punctured domain. We prove the pluricomplex Green function is $C^{1,\alpha}$ by constructing
approximate solutions and establishing uniform a priori estimates for the gradient and the complex Hessian. The singular solutions turn out to be smooth for the $k$-Hessian equations for $(n-1)$-$k$-admissible functions.
\end{abstract}

\keywords{\noindent\footnotesize Keywords: Pluricomplex Green function, Form type fully nonlinear ellptic equations, a priori estimates\\
MSC: 35J08, 35J60, 35J96}






\section{Introduction}

The Green functions play fundamental roles in the analysis of partial differential equations. On a bounded smooth domain $\Omega\subset\mathbb{R}^n$, the Green function $G(x,x_0)$ of the Laplace operator $-\Delta$ is the unique smooth solution to the homogeneous Dirichlet problem in the punctured domain
\begin{equation}
    \begin{cases}
        -\Delta G=0\quad{\rm in\;}\overline{\Omega}\setminus\{x_0\},\\
        G=0\quad{\rm on\;}\partial\Omega,\\
        G(x,x_0)\sim\Phi(x-x_0)\quad{\rm as\;}x\rightarrow x_0,
    \end{cases}
\end{equation}
for each fixed $x_0\in\Omega$. Here $\Phi$ is the fundamental solution of the Laplace equation defined by
\begin{equation}
    \Phi(x)=\begin{cases}
        \frac{1}{(n-2)|\mathbb{S}^{n-1}|}|x|^{2-n},\quad n\geq3,\\
        -\frac{1}{4\pi}\log|x|,\quad n=2,
    \end{cases}
\end{equation}
so that $-\Delta\Phi=\delta_0$ in the sense of distributions.

For the complex Monge-Amp\`{e}re operator
\begin{equation}\label{cma000}
    u\mapsto\det(u_{i\bar{j}}),
\end{equation}
the fundamental solution is $\Phi(z)=-\log|z|$, which is the radially symmetric plurisubharmonic function satisfying 
\begin{equation}\label{cma001}
    \det(\Phi_{i\bar{j}})=c\delta_0
\end{equation}
in the sense of pluripotential theory. We remark that unlike real Monge-Amp\`{e}re equations, plurisubharmonic solutions to \eqref{cma001} are not unique by the examples in Bedford-Taylor \cite{bedford1976dirichlet}. The pluricomplex Green function of operator \eqref{cma000} on $\Omega\subset\mathbb{C}^n$ solves the Dirichlet problem
\begin{equation}\label{pgfofcma}
    \begin{cases}
        \det(u_{i\bar{j}})=0 \quad{\rm in\;}\overline{\Omega}\setminus\{z_0\},\\
        u=0\quad{\rm on\;}\partial\Omega,\\
        u(z)=\log|z-z_0|+O(1)\quad{\rm as\;}z\rightarrow z_0.
    \end{cases}
\end{equation}
Klimek \cite{klimek1985extremal} proved the existence of the extremal function
\begin{equation}
    g_\Omega(z,z_0)=\sup\{v\in{\rm PSH}(\Omega):v<0,v(z)\leq\log|z-z_0|+O(1)\},\quad z,z_0\in\Omega.
\end{equation}
The regularity of the pluricomplex Green function on $\Omega\setminus\{z_0\}$ is closely related to the geometric conditions of $\partial\Omega$. Demailly \cite{demailly1985mesures} showed that $u(z)=g_\Omega(z,z_0)$ is unique and continuous as the solution to \eqref{pgfofcma} on hyperconvex domains. When $\Omega$ is strictly convex, Lempert \cite{lempert1981metrique,lempert1983solving} proved $u\in C^\infty(\overline{\Omega}\setminus\{z_0\})$. For $\Omega$ strongly pseudoconvex, Guan \cite{guan2000correction} and Blocki \cite{blocki2000c1} showed the $C^{1,1}$ regularity of $u$, which is also optimal by counterexamples constructed in Bedford-Demailly \cite{bedford1988two}.

The pluricomplex Green function for the complex Monge-Amp\`{e}re equation is also a natural generalization of the Green function in one-dimensional complex analysis. From this point of view, it exhibits many useful properties of pluriharmonic functions and  holomorphic maps. Lempert \cite{lempert1981metrique,lempert1983solving} used the pluricomplex Green function to study the Kobayashi distance. Bracci-Contreras-D\'{i}az-Madrigal \cite{bracci2009pluripotential} studied the semigroup of holomorphic automorphisms and proved a Schwarz type lemma for the pluricomplex Green function. The pluricomplex Poisson kernel was introduced in Bracci-Patrizio \cite{bracci2005monge}. Later, Bracci-Patrizio-Trapani \cite{bracci2009pluricomplex} showed the relationship between the pluricomplex Green function and the pluricomplex Poisson kernel and generalized the Phragmen-Lindel\"{o}f theorem, in which they also gave an explicit representation formula for all plurisubharmonic functions, combining the work of Demailly \cite{demailly1985mesures}.  

Recently, Gao-Ma-Zhang \cite{gao2023dirichlet} studied the pluricomplex Green function of the $k$-Hessian operator
\begin{equation}\label{ckh000}
    u\mapsto S_k(i\partial\bar{\partial} u)=\sum_{1\leq i_1<\cdots<i_k\leq n}\lambda_{i_1}\cdots\lambda_{i_k},
\end{equation}
where $1\leq k\leq n-1$ and $\lambda_i$ are the eigenvalues of $i\partial\bar{\partial}u$. In this case, the fundamental solution
$\Phi(z)=-|z|^{2-\frac{2n}{k}}$ is the radial $k$-admissible function satisfying $S_k(i\partial\bar{\partial}\Phi)=0$ for $z\neq0$. The Pluricomplex Green function of operator \eqref{ckh000} is the solution to
\begin{equation}\label{pgfofckh}
    \begin{cases}
        S_k(i\partial\bar{\partial}u)=0 \quad{\rm in\;}\overline{\Omega}\setminus\{z_0\},\\
        u=0\quad{\rm on\;}\partial\Omega,\\
        u(z)=-|z|^{2-\frac{2n}{k}}+O(1)\quad{\rm as\;}z\rightarrow z_0.
    \end{cases}
\end{equation}
Gao-Ma-Zhang \cite{gao2023dirichlet} proved the $C^{1,\alpha}$ regularity of $u$ on strongly pseudoconvex domains. The Dirichlet problem for homogeneous real $k$-Hessian equations on a punctured domain was solved by Gao-Ma-Zhang \cite{gao2023dirichletreal}, with which they proved a weighted geometric inequality for $(k-1)$-convex starshaped closed hypersurface in $\mathbb{R}^n$ for $\frac{n}{2}\leq k<n$. Using a B\^{o}cher type theorem by Labutin \cite{labutin2002potential}, the behavior of the solution near the singularity is uniquely determined in the real case. 

In this paper, we are interested in finding the pluricomplex Green function of the Monge-Amp\`{e}re equation for $(n-1)$-plurisubharmonic functions
\begin{equation}\label{n-1psheq}
    \left(\omega_0+\frac{1}{n-1}((\Delta_\omega u)\omega-i\partial\bar{\partial}u)\right)^n=e^h\omega^n\quad{\rm in\;}M,
\end{equation}
where $\omega_0,\omega$ are Hermitian metrics. Admissible solutions to \eqref{n-1psheq} are required to be $(n-1)$-plurisubharmonic, namely the sums of any $n-1$ eigenvalues of the complex Hessian of the solution are positive, in the sense of Harvey-Lawson \cite{harvey2012geometric,harvey2013p}. Equation \eqref{n-1psheq} was first investigated in Fu-Wang-Wu \cite{fu2010form,fu2015form}, where they assumed $(M,\omega)$ to be K\"{a}hler with non-negative orthogonal bisectional curvature. 
Since then, the so-called `form type' fully nonlinear elliptic equations have been extensively studied in both analysis and geometry. Tosatti-Weinkove \cite{tosatti2017monge} solved \eqref{n-1psheq} on general K\"{a}hler and Hermitian manifolds, which contributed to the resolution of the Gauduchon conjecture in Sz\'{e}kelyhidi-Tosatti-Weinkove \cite{szekelyhidi2017gauduchon}. Form type $k$-Hessian equations were considered in Sz\'{e}kelyhidi \cite{szekelyhidi2018fully} for $k$-positive real (1,1)-forms $\omega_0$. George-Guan-Qiu \cite{george2022fully} generalized these equations to $p$-plurisubharmonic functions, namely nonlinear equations of the sums of any $p$ eigenvalues.

Suppose $\Omega\subset\mathbb{C}^n$ is a bounded smooth domain and $u\in C^2(\Omega;\mathbb{R})$. Denote by $\mu[u]=(\mu_1,\cdots,\mu_n)\in\mathbb{R}^n$ the eigenvalues of $(\Delta u)\sum dz^i\wedge d\bar{z}^i-i\partial\bar{\partial}u$, and consider the equation 
\begin{equation}\label{equation0000intro}
    F(u_{i\bar{j}})=S_k(\mu[u])=h\geq0
\end{equation}
for any $k=1,\cdots,n$. We call \eqref{equation0000intro} the form type $k$-Hessian equation. Note that when $k=n$, \eqref{equation0000intro} is just \eqref{n-1psheq} for $M$ complex Euclidean. 

First of all we need to find the fundamental solution of equation \eqref{equation0000intro}, which is a radially symmetric admissible function satisfying $F(\Phi)=0$ in $\mathbb{C}^n\setminus\{0\}$. Similar to the Laplacian and complex $k$-Hessian case, we have $\Phi(z)=-|z|^{-\gamma}$ for some $\gamma>0$. An interesting phenomenon is that for $k$ not very large or small (in this case, for $1<k<n$), the fundamental solution is not unique. In fact, $\gamma$ may take two distinct positive values, one of which is independent of $k$. This particular choice of $\gamma$ occurs whenever $k$ is large ($k>1$), as it makes $\mu_i=0$ for $n-1$ out of $n$ $\mu_i$'s. Such behaviors also appear for more complicated form type equations, namely equations of the sum of any $p$ eigenvalues. 

Motivated by previous works, we solve the homogeneous Dirichlet problem for equation \eqref{equation0000intro} on punctured domains of $\mathbb{C}^n$. We prove the following theorems:

\begin{theorem}\label{maintheorem}
    Let $1\leq k\leq n-1$ and $\Omega\subset{\mathbb{C}^n}$ be a bounded smooth domain. Then for any $z_0\in\Omega$ and $\varphi\in C^\infty(\partial\Omega)$, there exists an admissible function $u\in C^\infty(\overline{\Omega}\setminus\{z_0\})$ such that
    \begin{equation}\label{maintheorem1}
    \begin{cases}
        S_k(\mu[u])=0 \quad{\rm in\;}\overline{\Omega}\setminus\{z_0\},\\
        u=\varphi\quad{\rm on\;}\partial\Omega,\\
        u(z)=-|z-z_0|^{-\gamma_k}+O(1)\quad{\rm as\;}z\rightarrow z_0,
    \end{cases}
\end{equation}
where $\gamma_k=\frac{2n^2-4n+2k}{n-k}>0$.
\end{theorem}

\begin{theorem}\label{maintheorem'}
    Let $2\leq k\leq n$ and $\Omega\subset{\mathbb{C}^n}$ be a bounded, smooth, 1-pseudoconvex domain. Then for any $z_0\in\Omega$, there exists an admissible function $u\in C^{1,\alpha}(\overline{\Omega}\setminus\{z_0\})$ for all $0<\alpha<1$ such that
    \begin{equation}\label{maintheorem2}
    \begin{cases}
         S_k(\mu[u])=0 \quad{\rm in\;}\overline{\Omega}\setminus\{z_0\},\\
        u=0\quad{\rm on\;}\partial\Omega,\\
        u(z)=-|z-z_0|^{4-2n}+O(1)\quad{\rm as\;}z\rightarrow z_0.
    \end{cases}
\end{equation}
Furthermore, if $2\leq k\leq n-1$, we have $u\in C^\infty(\overline{\Omega}\setminus\{z_0\})$ whenever $\Omega$ is bounded and smooth and $u|_{\partial\Omega}=\varphi\in C^\infty(\partial\Omega)$.
\end{theorem}

Note that when $k<n$, we obtain stronger results with weaker conditions. As we shall see in \eqref{uniformlyelliptic}, this is because equation \eqref{equation0000intro} for $k<n$ is `uniformly elliptic', which allows us to use Evans-Krylov theory \cite{evans1982CPAM,krylov1983EK} in the wake of second order a priori estimates and bootstrap to $C^\infty$, despite the degenerate right hand side and the weaker boundary conditions. When $k=n$, equation \eqref{equation0000intro} is in general not uniformly elliptic, and similar to the Monge-Amp\`{e}re and Hessian equations, the best regularity for the pluricomplex Green function one can hope for is $C^{1,1}$.

The main idea in proving Theorem \ref{maintheorem} and \ref{maintheorem'} is to establish uniform a priori estimates for the Dirichlet problems approximating \eqref{maintheorem1} and \eqref{maintheorem2}. Since the approximate solutions blow up near the singularity $z_0$, we cannot use the standard maximum principle directly. Instead, we need to determine the blow up rates of the gradient and the complex Hessian. The key argument is to construct barrier functions based on the blow up rates so that the test functions are uniformly bounded on the boundaries of the regions approximating $\Omega\setminus\{z_0\}$. Only then can we apply the maximum principle to the test functions and derive uniform bounds for the gradient and the complex Hessian on compact subsets of the punctured domain.  

The standard scheme to establish the existence and regularity for Dirichlet problems of fully nonlinear elliptic equations using the continuity method and the maximum principle was explored in the pioneering work of Caffarelli-Nirenberg-Spruck \cite{caffarelli1985dirichlet}. In general, one may consider on a complex Hermitian manifold $(M,g)$ with smooth boundary
\begin{equation}\label{fullynonlinearCNSintro}
\begin{cases}
    f(\lambda(\chi+i\partial\bar{\partial}u))=h\quad{\rm in\;}M,\\
    u=\varphi\quad{\rm on\;}\partial M,
\end{cases}
\end{equation}
where $\lambda$ are the eigenvalues of $\chi+i\partial\bar{\partial}u$ with respect to the metric $g$, and $f$ is a smooth symmetric function of $\lambda$ defined in an open convex symmetric cone $\Gamma$ with vertex at 0, satisfying certain structural conditions which we shall explain in Section 2. A function $u$ satisfying \eqref{fullynonlinearCNSintro} is called admissible, if $\lambda(i\partial\bar{\partial}u)\in\Gamma$. Under this general framework, Li \cite{li2004dirichlet} solved \eqref{fullynonlinearCNSintro} on $\Omega\subset\mathbb{C}^n$ assuming either existence of admissible subsolutions or geometric conditions of $\partial\Omega$. More recently, the problem was solved by Guan \cite{Guan2014,guan2023dirichlet} on Riemannian manifolds with boundary, and by Sz\'{e}kelyhidi \cite{szekelyhidi2018fully} on compact Hermitian manifolds, assuming the existence of what he calls $\mathcal{C}$-subsolutions.

Another intriguing question is to generalize the pluricomplex Green function to complex manifolds. Semmes \cite{semmes1992generalization} considered the pluricomplex Green function in Riemann mappings. Poletsky \cite{poletsky2020pluricomplex} studied the pluricomplex Green function on hyperconvex manifolds. L{\'a}russon-Sigurdsson \cite{larusson1999plurisubharmonic} treated the pluricomplex Green function on analytic subspaces of complex manifolds. On the other hand, the behavior of the homogeneous complex Monge-Amp\`{e}re equation on compact K\"{a}hler manifolds near singularities is closely related to birational geometry and the theory of K-stability \cite{donaldson2012kahler,guenancia2016conic,jeffres2016kahler}. Phong-Sturm \cite{phong2014singularities} prescibed the singularities of the pluricomplex Green function with isolated poles, and applied these regularity results in the construction of geodesic rays in the space of K\"{a}hler potentials, induced by some given test-configuration. More recently, Darvas-Di Nezza-Lu \cite{darvas2023relative} considered solutions of complex Monge-Amp\`{e}re equations with prescribed singularity profiles using pluripotential theory.

The rest of this paper is organized as follows. Section 2 recalls some background materials, including various notations and useful formulas in the a priori estimates. In Section 3, we construct the pluricomplex Green function, whose existence and regularity depend on the gradient and complex Hessian estimates which we prove in Section 4 and 5.

\section{Preliminaries}

\subsection{General theory}

First of all, we recall the fundamental assumptions on \eqref{fullynonlinearCNSintro} in the spirit of Caffarelli-Nirenberg-Spruck \cite{caffarelli1985dirichlet}. $\Gamma$ is an open convex symmetric cone with vertex at 0, and 
\begin{equation}
    \Gamma_n\subset\Gamma\subset\Gamma_1=\{S_1(\lambda)=\sum\lambda_i>0\}.
\end{equation}
$f$ satisfies 
\begin{enumerate}
    \item (Ellipticity) $f_i>0$ in $\Gamma$ for all $1\leq i\leq n$;
    \item (Concavity) $f$ is concave;
    \item (Non-degenerancy) $\sup_{\partial\Gamma}f<h<\sup_\Gamma f$;
    \item $\lim_{t\rightarrow+\infty}f(t\lambda)=\sup_\Gamma f$ for all $\lambda\in\Gamma$.
\end{enumerate}
The equation is called degenerate, if we allow $\sup_{\partial\Gamma}f=\inf h$. For Hessian type equations, the equation is called homogeneous when $h\equiv0=\sup_{\partial\Gamma_k}S_k$. 

We will need some general formulas for the derivatives of the eigenvalues $\lambda$ as well as the symmetric function $f$. As above, write 
\begin{equation}\label{fullynonlinearCNS}
    F(A):=f(\lambda_1,\cdots,\lambda_n)=h,
\end{equation}
where $A=(a_{i\bar{j}})$ is an Hermitian matrix with eigenvalues $\lambda(A)=(\lambda_1,\cdots,\lambda_n)$.
As is stated in \cite{spruck2005geometric}, the derivatives of $\lambda_i$ with respect to the entries $a_{p\bar{q}},a_{r\bar{s}}$ of $A$ at $A$ diagonal are
\begin{gather}
    \frac{\partial\lambda_i}{\partial a_{p\bar{q}}}:=\lambda_i^{p\bar{q}}=\delta_{ip}\delta_{iq},\\
    \frac{\partial^2\lambda_i}{\partial a_{p\bar{q}}\partial a_{r\bar{s}}}:=\lambda_i^{p\bar{q},r\bar{s}}=(1-\delta_{ip})\frac{\delta_{iq}\delta_{ir}\delta_{ps}}{\lambda_i-\lambda_p}+(1-\delta_{ir})\frac{\delta_{is}\delta_{ip}\delta_{rq}}{\lambda_i-\lambda_r}.
\end{gather}
The derivatives of $F$ with respect to $a_{i\bar{j}},a_{l\bar{m}}$ at $A$ diagonal are (see \cite{andrews1994contraction,gerhardt1996closed})
\begin{gather}
    F^{i\bar{j}}:=\frac{\partial F}{\partial a_{i\bar{j}}}=\delta_{ij}f_i,\\
    \label{2ndderivativeofF}F^{i\bar{j},l\bar{m}}:=\frac{\partial^2F}{\partial a_{i\bar{j}}\partial a_{l\bar{m}}}=f_{il}\delta_{ij}\delta_{lm}+\frac{f_i-f_j}{\lambda_i-\lambda_j}(1-\delta_{ij})\delta_{im}\delta_{jl}.
\end{gather}
Since $f$ is concave and symmetric, we always have $f_i\leq f_j$ whenever $\lambda_i\geq\lambda_j$.

\subsection{Form type $k$-Hessian equations}

Suppose $u$ is a smooth real function defined on some bounded smooth domain $\Omega\subset\mathbb{C}^n$. Denote by $\lambda=(\lambda_1,\cdots,\lambda_n)\in\mathbb{R}^n$ the eigenvalues of the complex Hessian $i\partial\bar{\partial}u$, and set 
\begin{equation}
    \mu_i=\sum_{j\neq i}\lambda_j=S_1(\lambda)-\lambda_i,\quad\mu[u]=(\mu_1,\cdots\mu_n)\in\mathbb{R}^n.
\end{equation} 
For any admissible function $u:\Omega\rightarrow\mathbb{R}$, write 
\begin{equation}\label{equation0000}
    F(u_{i\bar{j}})=f(\lambda)=S_k^{\frac{1}{k}}(\mu[u])=h\geq0,
\end{equation}
where $u$ is called $(n-1)$-$k$-admissible, or admissible in short, if $\mu[u]\in\Gamma_k$ is $k$-admissible, i.e.
\begin{equation}
    \mu[u]\in\Gamma_k=\{\mu\in\mathbb{R}^n:S_j(\mu)>0,j=1,\cdots,k\},
\end{equation}
where
\begin{equation}
    S_k(\mu):=\sum_{1\leq i_1<\cdots<i_k\leq n}\mu_{i_1}\cdots\mu_{i_k}
\end{equation}
is the $k$-th symmetric polynomial. When $k=n$, this corresponds to $(n-1)$-plurisubharmonic functions. It is easy to check that \eqref{equation0000} satisfies the fundamental assumptions as above.

In what follows, we will consider solutions of the homogeneous version (i.e. $h=0$) of equation \eqref{equation0000} with isolated singularities. Denote the linearized operator of $F$ at $u$ by $L=F^{i\bar{j}}\partial_{i\bar{j}}$.
Differentiating \eqref{equation0000}, we have
\begin{gather}
\label{linearizedderivative1}Lu_l=h_l,\quad Lu_{\bar{l}}=h_{\bar{l}},\\
\label{linearizedderivative2}Lu_{\gamma\bar{\gamma}}+F^{i\bar{j},l\bar{m}}u_{i\bar{j}\gamma}u_{l\bar{m}\bar{\gamma}}=h_{\gamma\bar{\gamma}}.
\end{gather}
where, in terms of matrices,
\begin{equation}
    F^{i\bar{j}}=\frac{\partial F}{\partial{u}_{i\bar{j}}}(u_{i\bar{j}})\geq0,\quad F^{i\bar{j},l\bar{m}}=\frac{\partial^2F}{\partial u_{i\bar{j}}\partial u_{l\bar{m}}}(u_{i\bar{j}})\leq0.
\end{equation}
This means that $L$ is elliptic and $F$ is concave. Indeed,  WLOG assume $u_{i\bar{j}}$ and $F^{i\bar{j}}$ are diagonal. Then 
\begin{gather}
      \label{lambdaandmu}\lambda_i=u_{i\bar{i}}=S_1(\lambda)-\mu_i=\frac{S_1(\mu)}{n-1}-\mu_i,\\
      F^{i\bar{i}}=\frac{\partial f}{\partial\lambda_i}=\sum_j\frac{\partial}{\partial\mu_j}(S_k^{\frac{1}{k}}(\mu))\frac{\partial\mu_j}{\partial\lambda_i}=\frac{1}{k}S_k^{\frac{1}{k}-1}(\mu)\sum_{j\neq i}S_{k-1;j}(\mu)\geq0,
\end{gather}
where
\begin{equation}
    S_{k-1;j}(\mu)=\frac{\partial}{\partial\mu_j}S_k(\mu)=S_{k-1}(\mu)\mid_{\mu_j=0}.
\end{equation} 
Similar to the complex Hessian equations, using \eqref{lambdaandmu}, \eqref{p04} and \eqref{p02} we still have
\begin{equation}\label{linearizedsum1}
    \sum_iF^{i\bar{i}}u_{i\bar{i}}=\frac{1}{k}S_k^{\frac{1}{k}-1}(\mu)\sum_i((n-k+1)S_{k-1}(\mu)-S_{k-1;i}(\mu))\left(\frac{S_1(\mu)}{n-1}-\mu_i\right)=h.
\end{equation}
We also adapt the notation 
\begin{equation}
    \mathcal{F}=\sum_iF^{i\bar{i}}=\frac{(n-1)(n-k+1)}{k}S_k^{\frac{1}{k}-1}(\mu)S_{k-1}(\mu)\geq C_{n,k}>0
\end{equation}
following \eqref{p01} and \eqref{p02}.
For clarity we assume $\lambda_1\leq\cdots\leq\lambda_n$ so that $\mu_1\geq\cdots\geq\mu_n$ with $S_{k-1;1}(\mu)\leq\cdots\leq S_{k-1;n}(\mu)$. In this case $F^{1\bar{1}}\geq\cdots\geq F^{n\bar{n}}\geq0$. A critical observation when $1\leq k\leq n-1$ is that for any $1\leq i\leq n$, we always have
\begin{equation}\label{uniformlyelliptic}
    F^{i\bar{i}}\geq\frac{1}{k}S_k^{\frac{1}{k}-1}(\mu)S_{k-1;k}(\mu)\geq c_0\mathcal{F}\geq c_0(n,k)>0
\end{equation}
for some fixed constant $c_0(n,k)>0$ independent of $f$, where we use \eqref{p02} and \eqref{p03}. Equation \eqref{equation0000} for $k<n$ is thus `uniformly elliptic' in a certain sense. Because of this, we may expect higher regularity for the pluricomplex Green function. 

On the other hand, when $k=n$, one usually write 
\begin{equation}\label{equation0000'}
    f(\lambda)=\log(\mu_1\cdots\mu_n)=\log h
\end{equation}
instead of \eqref{equation0000}, which boils down to the complex Monge-Amp\`{e}re equation for $(n-1)$ plurisubharmonic functions \eqref{n-1psheq} on bounded domains $\Omega\subset\mathbb{C}^n$. In this case, all $\mu_i>0$, and after diagonalizing we have
\begin{equation}\label{linearizedoperator'}
    F^{i\bar{i}}=f_i=\sum_{j\neq i}\frac{1}{\mu_j},\quad\mathcal{F}=(n-1)S_{-1}(\mu)
\end{equation}
where we write $S_{-1}(\mu)=\sum_i\mu_i^{-1}$. Similar to the complex Monge-Amp\`{e}re equation, we now have
\begin{equation}\label{linearizedsum1'}
    \sum_iF^{i\bar{i}}u_{i\bar{i}}=\sum_i\left(S_{-1}(\mu)-\frac{1}{\mu_i}\right)\left(\frac{S_1(\mu)}{n-1}-\mu_i\right)=n.
\end{equation}
One can also show that $F^{n\bar{n}}\geq\cdots\geq F^{2\bar{2}}\geq c_0\mathcal{F}$ for some $c_0>0$, assuming $\lambda_1\geq\cdots\geq\lambda_n$, but equation \eqref{equation0000'} is in general not uniformly elliptic.

\subsection{The fundamental solution}
We need to find the fundamental solution of the form type $k$-Hessian equation \eqref{equation0000}. Analogously, we have
\begin{defn}
    The fundamental solution $\Phi(z)$ of equation \eqref{equation0000} is a radially symmetric smooth admissible function defined on the punctured complex Euclidean space $\mathbb{C}^n\setminus\{0\}$ satisfying
    \begin{equation}
        S_k(\mu[\Phi])=0\quad{\rm in\;}\mathbb{C}^n\setminus\{0\}.
    \end{equation}
\end{defn}
To determind $\Phi$, set $\Phi(z)=\phi(|z|^2)$, where $\phi\in C^\infty(\mathbb{R}^+)$. Simple calculations imply
\begin{equation}
    \Phi_{i\bar{j}}(z)=\phi''\bar{z}_iz_j+\phi'\delta_{ij}.
\end{equation}
The eigenvalues of $i\partial\bar{\partial}\Phi$ are thus $\lambda=(\phi',\cdots,\phi',\phi'+\phi''|z|^2)$ and $\mu[\Phi]=(\mu_1,\cdots,\mu_n)$ with
\begin{equation}
    \mu_1=\cdots=\mu_{n-1}=(n-1)\phi'+\phi''|z|^2,\quad\mu_n=(n-1)\phi'.
\end{equation}
It follows that
\begin{equation}
\begin{aligned}
    S_k(\mu[\Phi])&=\binom{n-1}{k}((n-1)\phi'+\phi''|z|^2)^k+\binom{n-1}{k-1}((n-1)\phi'+\phi''|z|^2)^{k-1}(n-1)\phi'\\
    &=\frac{1}{k}\binom{n-1}{k-1}((n-1)\phi'+\phi''|z|^2)^{k-1}((n-k)\phi''|z|^2+n(n-1)\phi').
\end{aligned}
\end{equation}
Therefore, solving the ODE $|z|^2\phi''+(n-1)\phi'=0$ for $k>1$ and $(n-k)|z|^2\phi''+n(n-1)\phi'=0$, we have $\Phi(z)=C_1|z|^{-\gamma}+C_2$ for suitable $\gamma\in\mathbb{R}^+$. Note that we also need $C_1<0$ so that $\Phi$ is admissible. Without loss of generality we omit the constant term and normalize $C_1=-1$. We have

\begin{prop}\label{fundamentalsolution}
    The fundamental solution of equation \eqref{equation0000} is $\Phi(z)=-|z|^{-\gamma}$, where
    \begin{equation}
        \gamma=
        \begin{cases}
            2n-2, & k=1;\\
            \frac{2n^2-4n+2k}{n-k}{\rm\;or\;}2n-4,& 1<k<n;\\
            2n-4, & k=n.
        \end{cases}
    \end{equation}
\end{prop}
Note that when $k=1$, the fundamental solution $-|z|^{2-2n}$ in the sense of Proposition \ref{fundamentalsolution} is just the fundamental solution of the Laplace equation in $\mathbb{R}^{2n}$. When $\gamma=2n-4$, we always have $\mu_1=\cdots=\mu_{n-1}=0$, so that $S_k(\mu)=0$ whenever $k>1$. 

\subsection{Existence of subsolutions}\label{meanconvex}

Consider the non-degenerate Dirichlet problem of equation \eqref{equation0000} and \eqref{equation0000'} on a bounded smooth domain
\begin{equation}\label{dirichlet0000}
    \begin{cases}
        S_k(\mu[u])=h>0 \quad{\rm in\;}\Omega,\\
        u=\varphi\quad{\rm on\;}\partial\Omega,\\
    \end{cases}
\end{equation}
where $h,\varphi$ are smooth functions, and WLOG $0\in\Omega$. If $k<n$, since \eqref{equation0000} is uniformly elliptic, we know from the classical theory of fully nonlinear uniformly elliptic equations (see \cite{caffarelli1995fully}) that smooth solutions to \eqref{dirichlet0000} always exist. 

When $k=n$, we need a natural geometric condition on the boundary $\partial\Omega$ to construct subsolutions and solve problem \eqref{dirichlet0000}. 

\begin{defn}\label{pseudoconvexity}
Suppose $\Omega\subset\mathbb{C}^n$ is a bounded domain with smooth boundary. A defining function of $\Omega$ is a smooth function $\sigma\in C^\infty(\overline{\Omega})$ such that 
\begin{equation}
    \sigma<0{\;\rm in\;}\Omega{\;\rm and\;}\sigma=0,\nabla\sigma\neq0{\;\rm on\;}\partial\Omega.
\end{equation}


$\Omega$ is called a pseudo mean-convex (or 1-pseudoconvex) domain, if 
for any $z_0\in\partial\Omega$,  under holomorphic normal coordinates $(z_1,\cdots,z_n)$ where $x_n={\rm Re}(z_n)$ is the interior normal direction of $\partial\Omega$ at $z_0$, 
\begin{equation}
    {\rm tr}\sigma_{l\bar{m}}(z_0)\geq c_0>0,
\end{equation}
where $\sigma_{l\bar{m}}(z_0)=(\frac{\partial^2\sigma}{\partial z_l\partial\bar{z}_m}(z_0))_{1\leq l,m\leq n-1}$ is the Levi form of $\sigma$, namely the $(n-1)\times(n-1)$ complex Hessian matrix of $\sigma$ restricted to ${\rm Span}(z_1,\cdots,z_{n-1})$.
\end{defn}

With the extra condition on $\Omega$, consider the function
\begin{equation}
    \rho(z)=\sigma(z)+N\sigma^2(z)\quad{\rm on\;}\overline{\Omega},
\end{equation}
where $\sigma$ is a defining function of $\Omega$ and $N\gg1$ is a large constant. The Levi form of $\rho$ coincides with that of $\sigma$ on $\partial\Omega$, while $\rho_{n\bar{n}}(z_0)=N$ for all $z_0\in\partial\Omega$. Here $x_n$ is the interior normal direction of $\partial\Omega$ at $z_0$. By Lemma 2.1 in \cite{caffarelli1985dirichlet}, we have on $\partial\Omega$
\begin{equation}
\lambda_l(i\partial\bar{\partial}\rho)=\lambda_l(i\partial\bar{\partial}\sigma)+o(1),\quad\lambda_n(i\partial\bar{\partial}\rho)=N+O(1),\quad1\leq l\leq n-1.
\end{equation}
Thus for $N$ sufficiently large,
\begin{equation}\label{cal200}
\mu_l[\rho]=N+O(1)>0,\quad\mu_n[\rho]={\rm tr}\sigma_{l\bar{m}}+o(1).
\end{equation}
Therefore, we know $\mu[\rho]\in\Gamma_n$ on $\partial\Omega$ when $\Omega$ is pseudo mean-convex, and $S_n(\mu[\rho])\geq c_0N^{n-1}$ on $\partial\Omega$. By continuity, $S_n(\mu[\rho])\geq aN^{n-1}>0$ near $\partial\Omega$. 
Note that $\rho$ is also a defining function of $\Omega$ (at least up to subtracting a large positive cut-off function), so that $\rho=0$ on $\partial\Omega$ and $\rho<0$ in $\Omega$. Define the function
\begin{equation}
    \underline{u}=\varphi+\sup\{A\rho,B(|z|^2-d^2)\}\quad {\rm on}\;\overline{\Omega},
\end{equation}
where $d={\rm diam}(\Omega)$ and $A,B$ are positive constants. Clearly $\underline{u}=\varphi$ on $\partial\Omega$. If we choose $A\gg B$ very large, then $\underline{u}=\varphi+A\rho$ only very close to $\partial\Omega$, where we have $S_n(\mu[\underline{u}])\geq S_n(\mu[\varphi])+AaN^{n-1}>h$. Away from $\partial\Omega$ we always have $\underline{u}=\varphi+B(|z|^2-d^2)$, where we also have $S_n(\mu[\underline{u}])\geq S_n(\mu[\varphi])+B(n-1)^n>h$, provided we take $B$ sufficiently large. This means that $\underline{u}$ is a subsolution to problem \eqref{dirichlet0000}. It follows that smooth solutions to the Dirichlet problem exist. 

\begin{rem}
    For the complex Monge-Amp\`{e}re equation, it is well-known in \cite{caffarelli1985dirichlet} that the strong pseudoconvexity of $\Omega$ ensures the solvability of the Dirichlet problem. For complex $k$-Hessian equations, this natural condition is the so-called $(k-1)$-pseodoconvexity (see \cite{li2004dirichlet}), which means that the Levi form of the defining function has eigenvalues in $\Gamma_{k-1}$. These geometric conditions on the boundary are unified to `$\Gamma$-pseudoconvexity'
    in \cite{li2004dirichlet}, which varies with the admissible cone $\Gamma$ for different nonlinear equations. It is also worth noting that here when $k<n$, any bounded smooth domain is $\Gamma$-pseudoconvex, namely no extra geometric conditions are needed to solve the Dirichlet problem. This can be seen from the same construction of subsolutions as the $k=n$ case; the only difference is that \eqref{cal200} is already enough for $\mu[\rho]\in\Gamma_k$ for $k<n$ regardless of the boundary.

\end{rem}
\subsection{Symmetric polynomials}
We will need the following fundamental properties of symmetric polynomials in the proof of a priori estimates. Suppose $\mu\in\Gamma_k$ and $\mu_1\geq\cdots\geq\mu_n$. The $k$-symmetric polynomials are denoted by $S_k=S_k(\mu)$ with $S_0=1$ and $S_l=0$ for all $l>n$. For every $1\leq k,i\leq n$, we always have
\begin{gather}\label{p04}
    S_k=S_{k;i}+\mu_iS_{k-1;i},\\
\label{p01}
    S_k^{\frac{1}{k}}\leq C_{n,k}S_{k-1}^{\frac{1}{k-1}},\\
\label{p02}
    \sum_{i=1}^nS_{k-1;i}=(n-k+1)S_{k-1},\\
\label{p03}
    S_{k-1;k}\geq C_{n,k}\sum_{i=1}^nS_{k-1;i}.
\end{gather}

We refer the readers to \cite{wang2009k} and the references therein for more of the same as well as their proofs.

\section{The pluricomplex Green function}

Let $k<n$ and $\Omega\subset{\mathbb{C}^n}$ be a bounded smooth domain. WLOG assume $z_0=0\in B_1(0)\subset\Omega$. As is discussed in Section \ref{meanconvex}, there exists an admissible solution $v\in C^\infty(\overline{\Omega})$ to the Dirchlet problem
\begin{equation}
    \begin{cases}
        S_k(\mu[v])=1 \quad{\rm in\;}\Omega,\\
        v=\varphi-\Phi\quad{\rm on\;}\partial\Omega.\\
    \end{cases}
\end{equation}
Set $\underline{u}=v+\Phi\in C^\infty(\overline{\Omega}\setminus\{0\})$. The concavity of equation \eqref{equation0000} and the maximum principle imply that $\underline{u}$ satisfies
\begin{equation}
    \begin{cases}
        S_k(\mu[\underline{u}])\geq1 \quad{\rm in\;}\Omega\setminus\{0\},\\
        \underline{u}=\varphi\quad{\rm on\;}\partial\Omega,\\
        \underline{u}\leq\varphi\quad{\rm in\;}\Omega\setminus\{0\},\\
        \underline{u}(z)=\Phi(z)+O(1)\quad{\rm as\;}z\rightarrow0.
    \end{cases}
\end{equation}

Fix some $\varepsilon_0\in(0,1]$. For $\varepsilon\in(0,\varepsilon_0)$ sufficiently small, set $\Omega_\varepsilon=\Omega\setminus\overline{B_\varepsilon(0)}$ and consider the Dirichlet problem
\begin{equation}\label{approxdirichletproblem}
    \begin{cases}
        S_k(\mu[u^\varepsilon])=\varepsilon \quad{\rm in\;}\Omega_\varepsilon,\\
        u^\varepsilon=\underline{u}\quad{\rm on\;}\partial\Omega_\varepsilon.\\
    \end{cases}
\end{equation}
Note that $\underline{u}$ is now a subsolution. We know from \cite{george2022fully} that there exists a unique admissible solution $u^
\varepsilon\in C^\infty(\overline{\Omega}_\varepsilon)$ of equation \eqref{approxdirichletproblem}. The maximum principle implies that $u^\varepsilon$ is monotone decreasing in $\varepsilon$, namely for any $0<\varepsilon'\leq\varepsilon<\varepsilon_0$, 
\begin{equation}
    \underline{u}\leq u^\varepsilon\leq u^{\varepsilon'}\leq\Phi+C\quad{\rm in\;}\Omega_\varepsilon.
\end{equation}
Thus the pointwise limit
\begin{equation}
    u(z):=\lim_{\varepsilon\rightarrow0}u^\varepsilon(z)
\end{equation}
exists for all $z\in\overline{\Omega}\setminus\{0\}$. 

It suffices to establish the regularity results for $u$ as stated in Theorem \ref{maintheorem}. To this end, we will need a priori estimates of $u^\varepsilon$ up to second order:
\begin{theorem}\label{aprioriestimates}
    For $1\leq k\leq n-1$ in equation \eqref{approxdirichletproblem}, there exists a constant $C>0$ independent of $\varepsilon$ such that in $\overline{\Omega}_\varepsilon$,
    \begin{equation}\label{C0estimate}
        |u^\varepsilon(z)|\leq C|z|^{-\gamma},
    \end{equation}
    \begin{equation}\label{C1estimate}
        |\nabla u^\varepsilon(z)|\leq C|z|^{1-\gamma}(|z|^2-\varepsilon^2)^{-1},
    \end{equation}
    \begin{equation}\label{C2estimate}
        |u^\varepsilon_{i\bar{j}}(z)|\leq C|z|^{2-\gamma}(|z|^2-\varepsilon^2)^{-2}.
    \end{equation}
\end{theorem}

Once Theorem \ref{aprioriestimates} is justified, by \eqref{uniformlyelliptic} and \eqref{C2estimate}, equation \eqref{approxdirichletproblem} is uniformly elliptic at any solution $u^\varepsilon$, with ellipticity constants independent of $\varepsilon$. By Evans-Krylov theory \cite{evans1982CPAM,krylov1983EK}, we gain an extra $C^{\alpha}$ regularity for $i\partial\bar{\partial}u^\varepsilon$ in $K$, which means that $u^\varepsilon$ is uniformly bounded in $C^{2,\alpha}(K)$. Using classical Schauder theory, we see $u\in C^{\infty}(\overline{\Omega}\setminus\{0\})$ is an admissible solution to the Dirichlet problem in a punctured domain
\begin{equation}\label{punctureddirichletproblem}
    \begin{cases}
        S_k(\mu[u])=0 \quad{\rm in\;}\overline{\Omega}\setminus\{0\},\\
        u=\varphi\quad{\rm on\;}\partial\Omega,\\
        u(z)=\Phi(z)+O(1)\quad{\rm as\;}z\rightarrow0.
    \end{cases}
\end{equation}

We call this $u$ the pluricomplex Green function for equation \eqref{equation0000}.

Note that the above constructions also work for the $k=n$ case, except that we now need $\varphi$ constant and the boundary of $\Omega$ pseudo mean-convex for the existence of $v$ and to work out the following a priori estimates:

\begin{theorem}\label{aprioriestimatesfork=n}
    For $k=n,\varphi=0$ and $\Omega$ 1-pseudoconvex in equation \eqref{approxdirichletproblem}, there exists a constant $C>0$ independent of $\varepsilon$ such that in $\overline{\Omega}_\varepsilon$,
    \begin{equation}\label{C0estimate'}
        |u^\varepsilon(z)|\leq C|z|^{-\gamma},
    \end{equation}
    \begin{equation}\label{C1estimate'}
        |\nabla u^\varepsilon(z)|\leq C|z|^{-\gamma-1},
    \end{equation}
    \begin{equation}\label{C2estimate'}
        |u^\varepsilon_{i\bar{j}}(z)|\leq C|z|^{-\gamma-2}.
    \end{equation}
\end{theorem}

If \eqref{C2estimate'} holds, then for any compact set $K\subset\overline{\Omega}\setminus\{0\}$ and $\varepsilon>0$ sufficiently small so that $K\subset\overline{\Omega}_\varepsilon$, we have
\begin{equation}
    \|\Delta u^\varepsilon\|_{L^\infty(K)}=\|{\rm tr}(u^\varepsilon_{i\bar{j}})\|_{L^\infty(K)}\leq C_K
\end{equation}
for some constant $C_K>0$ independent of $\varepsilon$. By $W^{2,p}$ theory and Sobolev embedding, we know $u^\varepsilon$ is uniformly bounded in $C^{1,\alpha'}(K)$ for any $0<\alpha'<1$. By the Arzel\`{a}-Ascoli theorem, we assert that for all $0<\alpha<1$, up to a subsequence $u^\varepsilon$ converges to some $u$ in $C^{1,\alpha}(K)$, and that $u\in C^{1,\alpha}(\overline{\Omega}\setminus\{0\})$ is an admissible solution to \eqref{punctureddirichletproblem} for $k=n$ and $\varphi=0$.

The $C^0$ estimates in Theorem \ref{aprioriestimates} and \ref{aprioriestimatesfork=n} follow readily by construction. Indeed, we have
\begin{equation}
    \underline{u}\leq u^\varepsilon\leq\Phi+C_0,
\end{equation}
where $C_0=\|v\|_{C^0(\overline{\Omega})}>0$ is a uniform constant in $\varepsilon$. Thus, in $\overline{\Omega}_\varepsilon$,
\begin{equation}
    |u^\varepsilon(z)|\leq|z|^{-\gamma}+C_0\leq C|z|^{-\gamma}.
\end{equation}
In the next two sections, we move on to the gradient and the complex Hessian estimates for both cases. 

Before that, we would like to specify some useful notations. The $k=n$ case requires establishing derivative estimates on $\partial B_\varepsilon$. To this end, we need to apply a scaling argument by setting $\tilde{u}(z)=\varepsilon^{\gamma}u^\varepsilon(\varepsilon z)$ and $\tilde{\underline{u}}(z)=\varepsilon^{\gamma}\underline{u}(\varepsilon z)$ for $z\in D\setminus B_1$, where $D=\{z\in\mathbb{C}^n:\varepsilon z\in\Omega\}$. Then
\begin{equation}
    \begin{cases}
     S_k(\mu[u])=\varepsilon^{1+(\gamma+2)k}\quad {\rm{in}}\;D\setminus B_1,\\
     \tilde{u}=\tilde{\underline{u}}\quad {\rm{on}}\;\partial(D\setminus B_1).
    \end{cases}
\end{equation}
and
\begin{equation}
    \tilde{\underline{u}}(z)=\varepsilon^{\gamma}v(\varepsilon z)-|z|^{-\gamma}\quad{\rm in }\;D\setminus B_1.
\end{equation}
Clearly we have for any $z\in \overline{D}\setminus B_1$ and $w=\varepsilon z\in\Omega_\varepsilon$,
\begin{equation}\label{cal007}
    \nabla\tilde{u}(z)=\varepsilon^{1+\gamma}\nabla u(w),\quad\partial\bar{\partial}\tilde{u}(z)=\varepsilon^{2+\gamma}\partial\bar{\partial}u(w);
\end{equation}and for any $j\in\mathbb{N}$,
\begin{equation}\label{estimatesforsubsolution}
    |\nabla^j\tilde{\underline{u}}(z)|\leq\varepsilon^{\gamma+j}\|v\|_{C^j(\overline{\Omega})}+C_{n,k}|z|^{-\gamma-j}\leq C(\varepsilon^{\gamma+j}+1)\leq C.
\end{equation}

\section{Gradient estimates}

\subsection{Estimates on $\partial\Omega$}

WLOG assume $B_{r_0}(0)\subset\Omega$ for some fixed $r_0>0$. Taking $0<\varepsilon\ll r_0$, we have $\Omega_\varepsilon\subset\Omega_{r_0}\subset\Omega$. Let $h$ be the smooth harmonic function on $\Omega_{r_0}$ with boundary value
\begin{equation}
    h|_{\partial\Omega}=\varphi,\quad h|_{B_{r_0}}=-r_0^{2p}+C_0.
\end{equation}
By construction, we have
\begin{equation}
    \begin{cases}
\Delta u^\varepsilon=\frac{1}{n-1}S_1(\mu[u])>0=\Delta h\quad{\rm in\;}\Omega_{r_0},\\
        u^\varepsilon=h=\varphi\quad{\rm on\;}\partial\Omega,\\
        u^\varepsilon\leq\Phi+C_0=h\quad{\rm on\;}\partial B_{r_0}.
    \end{cases}
\end{equation}
The maximum principle then implies that $\underline{u}\leq u^\varepsilon\leq h$ in $\Omega_{r_0}$. Hence we have 
\begin{equation}\label{cal006}
    0\leq h_\nu\leq u^\varepsilon_\nu\leq\underline{u}_\nu\leq C\quad{\rm on\;}\partial\Omega.
\end{equation}
Furthermore, note that when $k=n$ and $\varphi=0$, by Hopf lemma we have
\begin{equation}\label{cal006'}
u_\nu^\varepsilon\geq h_\nu=c_1>0,
\end{equation}
which is critical in the second order boundary estimates.

\subsection{Global gradient estimates, $k<n$}
For brevity, write $u^\varepsilon=u$. The maximum principle tells us that 
\begin{equation}
    \max_{\overline{\Omega}_\varepsilon}u^\varepsilon\leq\max_{\partial\Omega_\varepsilon}\underline{u}\leq\max_{\partial\Omega}\varphi\leq C,
\end{equation}
which is a fixed constant independent of $\varepsilon$. WLOG we assume $u\leq C<0$ in $\overline{\Omega}_\varepsilon$ (else use $u-C$ instead, since interior estimates do not involve $\varphi$). Consider the auxiliary function 
\begin{equation}\label{generalP}
    P(z)=e^{\phi(u)+\psi(|z|^2-\varepsilon^2)}|\nabla u|^2.
\end{equation}
First we show $P$ is bounded in $\Omega_\varepsilon$. Indeed, suppose $P$ attains its maximum at $z_0\in\Omega_\varepsilon$. WLOG, we assume $u_{i\bar{j}}(z_0)$ is diagonal and so is $F^{i\bar{j}}(z_0)$, namely $u_{i\bar{i}}=\lambda_i$ and $F^{i\bar{i}}=f_i>0$ at $z_0$. The derivative tests tell us at $z_0$,
\begin{equation}\label{gradienttestofP}
    0=\partial_i\log P=\phi'u_i+\psi'\bar{z}_i+\frac{\partial_i|\nabla u|^2}{|\nabla u|^2},\quad0=\partial_{\bar{i}}\log P=\phi'u_{\bar{i}}+\psi'z_i+\frac{\partial_{\bar{i}}|\nabla u|^2}{|\nabla u|^2},
\end{equation}
and, using \eqref{gradienttestofP}, \eqref{linearizedsum1}, \eqref{linearizedderivative1},
\begin{equation}\label{2ndderivativetestofP}
\begin{aligned}
        0&\geq F^{i\bar{i}}\partial_{i\bar{i}}\log P\\
        &=\phi' F^{i\bar{i}}u_{i\bar{i}}+\phi''F^{i\bar{i}}|u_i|^2+\psi'\mathcal{F}+\psi''F^{i\bar{i}}|z_i|^2+\frac{F^{i\bar{i}}\partial_{i\bar{i}}|\nabla u|^2}{|\nabla u|^2}-F^{i\bar{i}}|\phi'u_i+\psi'\bar{z}_i|^2\\
        &\geq\phi'\varepsilon^{\frac{1}{k}}+\psi'\mathcal{F}+(\phi''-(1+\delta)|\phi'|^2)F^{i\bar{i}}|u_i|^2+(\psi''-(1+\frac{1}{\delta})|\psi'|^2)F^{i\bar{i}}|z_i|^2+\frac{F^{i\bar{i}}U_i^2+F^{i\bar{i}}\lambda_i^2}{|\nabla u|^2},
\end{aligned}
\end{equation}
where we denote $U_i=(\sum_l|u_{il}|^2)^{\frac{1}{2}}>0$. Rewriting \eqref{gradienttestofP} as
\begin{equation}
    \lambda_iu_i+u_{il}u_{\bar{l}}=-|\nabla u|^2(\phi'u_i+\psi'\bar{z}_i)
\end{equation}
and using Cauchy-Schwarz inequality, we have
\begin{equation}
\begin{aligned}
    (\lambda_i^2+U_i^2)(|u_i|^2+|\nabla u|^2)&\geq(|\lambda_iu_i|+U_i|\nabla u|)^2\geq(|\lambda_iu_i|+|u_{il}u_{\bar{l}}|)^2\\
    &\geq|\nabla u|^4|\phi'u_i+\psi'\bar{z}_i|^2\geq|\nabla u|^4((1-\delta)|\phi'|^2|u_i|^2-\frac{1-\delta}{\delta}|\psi'|^2|z_i|^2)
\end{aligned}
\end{equation}
for any $0<\delta<1$. Thus
\begin{equation}
    \frac{F^{i\bar{i}}U_i^2+F^{i\bar{i}}\lambda_i^2}{|\nabla u|^2}\geq\frac{1-\delta}{2}|\phi'|^2F^{i\bar{i}}|u_i|^2-\frac{1-\delta}{\delta}|\psi'|^2F^{i\bar{i}}|z_i|^2,
\end{equation}
and \eqref{2ndderivativetestofP} is reduced to
\begin{equation}
    0\geq\phi'\varepsilon^{\frac{1}{k}}+\psi'\mathcal{F}+(\phi''-\frac{1+3\delta}{2}|\phi'|^2)F^{i\bar{i}}|u_i|^2+(\psi''-\frac{2}{\delta}|\psi'|^2)F^{i\bar{i}}|z_i|^2.
\end{equation}
Take $\phi(u)=-a\log(-u),\psi(|z|^2-\varepsilon^2)=2\log(|z|^2-\varepsilon^2)$ for $a=2-\frac{2}{\gamma}\in(0,2)$. Since $\phi',\psi'>0$, we have
\begin{equation}\label{cal100}
    (a-\frac{1+3\delta}{2}a^2)\frac{F^{i\bar{i}}|u_i|^2}{u^2}\leq(2+\frac{8}{\delta})\frac{F^{i\bar{i}}|z_i|^2}{(|z|^2-\varepsilon^2)^2}\leq(2+\frac{8}{\delta})\mathcal{F}\frac{|z_0|^2}{(|z_0|^2-\varepsilon^2)^2}.
\end{equation}
Note that when $k<n$, we always have \eqref{uniformlyelliptic}, namely $F^{i\bar{i}}\geq c_0\mathcal{F}$ for all $1\leq i\leq n$. It follows that
\begin{equation}
    |\nabla u(z_0)|^2\leq C_{n,k,\delta}\frac{|z_0|^2}{(|z_0|^2-\varepsilon^2)^2}|u(z_0)|^2,
\end{equation}
provided that $0<\delta<\frac{1}{3(\gamma-1)}$ is sufficiently small so that the coefficient on the LHS of \eqref{cal100} is positive. Therefore, for any $z\in\Omega_\varepsilon$, we have
\begin{equation}\label{cal101}
\begin{aligned}
    \log P(z)&\leq\log P(z_0)\leq-a\log|u(z_0)|+2\log(|z_0|^2-\varepsilon^2)+\log|\nabla u(z_0)|^2\\
    &\leq C+(2-a)\log|u(z_0)|+\log|z_0|^2\leq C,
\end{aligned}
\end{equation}
the last inequality resulting from the $C^0$ estimate \eqref{C0estimate}. Clearly $P(z)=(-u)^{-a}(|z|^2-\varepsilon^2)^2|\nabla u|^2=0$ on $\partial B_\varepsilon$, and $P(z)$ is bounded on $\partial\Omega$ because of \eqref{cal006}. Combining these with \eqref{cal101} we prove \eqref{C1estimate}.

\subsection{Global gradient estimates, $k=n$}
Here we use the classical argument first proposed in \cite{guan2000correction}. (See also \cite{gao2023exterior} for similar estimates for complex Hessian equations.) Write $u^\varepsilon=u\leq0$ and set
\begin{equation}\label{P'}
    P(z)=(-u)^{-\beta}|\nabla u|^2,
\end{equation}where $\beta=2+\frac{2}{\gamma}=2+\frac{1}{n-2}>2$. We must show $P\leq C$ on $\overline{\Omega}_\varepsilon$. If $P$ attains an maximum at an interior point $z_0\in\Omega_\varepsilon$, then likewise we have at $z_0$, after diagonalization
\begin{equation}\label{gradienttestofP'}
    0=\partial_i\log P=-\beta\frac{u_i}{u}+\frac{\partial_i|\nabla u|^2}{|\nabla u|^2},\quad0=\partial_{\bar{i}}\log P=-\beta\frac{u_{\bar{i}}}{u}+\frac{\partial_{\bar{i}}|\nabla u|^2}{|\nabla u|^2},
\end{equation}
and
\begin{equation}\label{2ndderivativetestofP'}
\begin{aligned}
        0&\geq F^{i\bar{i}}\partial_{i\bar{i}}\log P=-\beta F^{i\bar{i}}\left(\frac{u_{i\bar{i}}}{u}-\frac{|u_i|^2}{u^2}\right)+F^{i\bar{i}}\left(\frac{\partial_{i\bar{i}}|\nabla u|^2}{|\nabla u|^2}-\frac{\partial_i|\nabla u|^2\partial_{\bar{i}}|\nabla u|^2}{|\nabla u|^4}\right)\\
        &=-\frac{\beta n}{u}+\beta\frac{F^{i\bar{i}}|u_i|^2}{u^2}+\frac{F^{i\bar{i}}|u_{il}|^2+F^{i\bar{i}}u_{i\bar{i}}^2}{|\nabla u|^2}-F^{i\bar{i}}\frac{|u_{il}u_{\bar{l}}|^2+2u_{i\bar{i}}{\rm Re}(u_{il}u_{\bar{i}}u_{\bar{l}})+u_{i\bar{i}}^2|u_i|^2}{|\nabla u|^4}.
\end{aligned}
\end{equation}
Note that by \eqref{gradienttestofP'} we actually have
\begin{equation}\label{cal102}
    \sum_{l=1}^nu_{il}u_{\bar{i}}u_{\bar{l}}=u_{\bar{i}}(\partial_i|\nabla u|^2-u_iu_{i\bar{i}})=\left(\frac{\beta|\nabla u|^2}{u}-u_{i\bar{i}}\right)|u_i|^2\in\mathbb{R},
\end{equation}
and substituting \eqref{gradienttestofP'} into \eqref{2ndderivativetestofP'} yields
\begin{equation}\label{cal105}
\begin{aligned}
    0&\geq|\nabla u|^2(F^{i\bar{i}}U_i^2+F^{i\bar{i}}\lambda_i^2)-(1-\frac{1}{\beta})(F^{i\bar{i}}|\sum_lu_{il}u_{\bar{l}}|^2+2F^{i\bar{i}}\lambda_i\sum_lu_{il}u_{\bar{i}}u_{\bar{l}}+F^{i\bar{i}}\lambda_i^2|u_i|^2)\\
    &=\sum_i\left((\sum_jF^{j\bar{j}}\lambda_j^2)|u_i|^2+|\nabla u|^2F^{i\bar{i}}U_i^2-(1-\frac{1}{\beta})F^{i\bar{i}}(|\sum_lu_{il}u_{\bar{l}}|^2+2\lambda_i\sum_lu_{il}u_{\bar{i}}u_{\bar{l}}+\lambda_i^2|u_i|^2)\right),
\end{aligned}
\end{equation}
where $U_i=(\sum_l|u_{il}|^2)^{\frac{1}{2}}$.

Let $\Lambda^+=\{1\leq i\leq n:\lambda_i\geq0\}$, $\Lambda^-=\{1\leq i\leq n:\lambda_i<0\}$.
For $i\in\Lambda^+,\lambda_i\geq0$, by \eqref{cal102} we have $\lambda_i\sum_lu_{il}u_{\bar{i}}u_{\bar{l}}\leq0$. In this case, by Cauchy-Schwarz inequality, for each fixed $i$ we have
\begin{gather}
     \label{cal103}|\nabla u|^2U_i^2=\left(\sum_l|u_l|^2\right)\left(\sum_l|u_{il}|^2\right)\geq|\sum_lu_{il}u_{\bar{l}}|^2,\\
     (\sum_jF^{j\bar{j}}\lambda_j^2)|u_i|^2\geq F^{i\bar{i}}\lambda_i^2|u_i|^2.
\end{gather}
This shows that
\begin{equation}\label{cal104}
   \sum_{i\in\Lambda^+}\left((\sum_jF^{j\bar{j}}\lambda_j^2)|u_i|^2+|\nabla u|^2F^{i\bar{i}}U_i^2-F^{i\bar{i}}(|\sum_lu_{il}u_{\bar{l}}|^2+2\lambda_i\sum_lu_{il}u_{\bar{i}}u_{\bar{l}}+\lambda_i^2|u_i|^2)\right)\geq0.
\end{equation}
So it suffices to consider the summation when $i\in\Lambda^-$. Plugging in \eqref{cal104} into \eqref{cal105} and using \eqref{cal103}, we get
\begin{equation}
\begin{aligned}
        0&\geq\sum_{i\in\Lambda^-}\left((\sum_jF^{j\bar{j}}\lambda_j^2)|u_i|^2+|\nabla u|^2F^{i\bar{i}}U_i^2-(1-\frac{1}{\beta})F^{i\bar{i}}(|\sum_lu_{il}u_{\bar{l}}|^2+2\lambda_i\sum_lu_{il}u_{\bar{i}}u_{\bar{l}}+\lambda_i^2|u_i|^2)\right)\\
        &\geq\sum_{i\in\Lambda^-}\left((\sum_jF^{j\bar{j}}\lambda_j^2)|u_i|^2+F^{i\bar{i}}(\frac{1}{\beta}|\sum_lu_{il}u_{\bar{l}}|^2-2(1-\frac{1}{\beta})\lambda_i\sum_lu_{il}u_{\bar{i}}u_{\bar{l}}-(1-\frac{1}{\beta})\lambda_i^2|u_i|^2)\right)\\
        &\geq\sum_{i\in\Lambda^-}(\sum_jF^{j\bar{j}}\lambda_j^2)|u_i|^2-(\beta-1)F^{i\bar{i}}\lambda_i^2|u_i|^2\\
        &=\sum_{i\in\Lambda^-}(\sum_{j\neq i}F^{j\bar{j}}\lambda_j^2-\frac{1}{n-2}F^{i\bar{i}}\lambda_i^2)|u_i|^2,
\end{aligned}
\end{equation}
where in the last inequality we use
\begin{equation}
    \frac{1}{\beta}|\sum_lu_{il}u_{\bar{l}}|^2+\frac{(\beta-1)^2}{\beta}\lambda_i^2|u_i|^2\geq2(1-\frac{1}{\beta})|\lambda_i\sum_lu_{il}u_{\bar{i}}u_{\bar{l}}|
\end{equation}
for each fixed $i$. WLOG $\lambda_1\geq\cdots\geq\lambda_n$ so that $\lambda_n<0$ and $F^{i\bar{i}}\lambda_i^2\leq F^{n\bar{n}}\lambda_n^2$ for any $i\in\Lambda^-$. Now it suffices to show that
\begin{equation}
    \sum_{j\neq n}F^{j\bar{j}}\lambda_j^2\geq\frac{1}{n-2}F^{n\bar{n}}\lambda_n^2.
\end{equation}
Note that by Cauchy-Schwarz inequality and \eqref{linearizedsum1'}, we have
\begin{equation}\label{cal106}
    (\sum_{j\neq n}F^{j\bar{j}})(\sum_{j\neq n}F^{j\bar{j}}\lambda_j^2)\geq(\sum_{j\neq n}F^{j\bar{j}}\lambda_j)^2=(n-F^{n\bar{n}}\lambda_n)^2.
\end{equation}
Using \eqref{linearizedoperator'} we have
\begin{equation}\label{cal107}
    \sum_{j\neq n}F^{j\bar{j}}=(n-2)S_{-1}(\mu)+\frac{1}{\mu_n}=(n-2)F^{n\bar{n}}+\frac{n-1}{\mu_n}<(n-2)F^{n\bar{n}}+\frac{n-1}{|\lambda_n|}
\end{equation}
since $\mu_n=\sigma_1(\lambda)-\lambda_n>-\lambda_n$. Combining \eqref{cal106} and \eqref{cal107}, we only need
\begin{equation}
    (n-F^{n\bar{n}}\lambda_n)^2\geq((n-2)F^{n\bar{n}}-\frac{n-1}{\lambda_n})(\frac{1}{n-2}F^{n\bar{n}}\lambda_n^2)
\end{equation}
which is equivalent to
\begin{equation}
    n^2\geq(2n-\frac{n-1}{n-2})F^{n\bar{n}}\lambda_n.
\end{equation}
This is clearly true for $n\geq3$ since $\lambda_n<0$. 

Combining the interior estimates and \eqref{cal006}, we still need to bound $P(z)$ on $\partial B_\varepsilon$, namely showing that for $|z|=\varepsilon$, 
    \begin{equation}
        |\nabla u^\varepsilon(z)|\leq C\varepsilon^{-\gamma-1}
    \end{equation}
for some constant $C>0$ independent of $\varepsilon$.
Indeed, let $\tilde{h}\in C^\infty(D\setminus B_1)$ solve the Dirichlet problem
\begin{equation}
    \begin{cases}
        \Delta\tilde{h}=0\quad{\rm in}\;D\setminus B_1,\\
        \tilde{h}=\tilde{u}=\tilde{\underline{u}}\quad {\rm{on}}\;\partial(D\setminus B_1).
    \end{cases}
\end{equation}Thus $S_1(i\partial\bar{\partial}\tilde{u})>0=S_1(i\partial\bar{\partial}\tilde{h})$ in $D\setminus B_1$. The maximum principle then asserts that $\tilde{\underline{u}}\leq\tilde{u}\leq\tilde{h}$ in $D\setminus B_1$. Therefore, $-C\leq\tilde{h}_\nu\leq\tilde{u}_\nu\leq\tilde{\underline{u}}_\nu\leq C$ on $\partial B_1$, and it follows from \eqref{cal007} that $|u^\varepsilon_\nu|\leq C\varepsilon^{-\gamma-1}$ on $\partial B_\varepsilon$.

\section{Second order estimates}

\subsection{Estimates on $\partial\Omega$}

We claim that there exists a constant $C>0$ independent of $\varepsilon$ such that on $\partial\Omega$,
\begin{equation}\label{C2boundary1}
    |u^\varepsilon_{i\bar{j}}(z)|\leq C.
\end{equation}

For brevity, write $u^\varepsilon=u$. Fix any $z_0\in\partial\Omega$. Up to a translation and rotation, we assume $z_0=0$ and that the positive $x_n$ axis is its interior normal. 
Write $z=(t',x_n)\in\mathbb{R}^{2n}$, where $t'=(t_1,\cdots,t_{2n-1}), t_{2i-1}=y_i,t_{2i}=x_i,1\leq i\leq n-1,t_{2n-1}=y_n$. Locally near 0, $\partial\Omega$ is represented by
\begin{equation}
    x_n=\rho(t')=\sum_{\alpha,\beta=1}^{2n-1}\rho_{\alpha\beta}(0)t_\alpha t_\beta+O(|t'|^3).
\end{equation}

\subsubsection{Tangential derivatives}
Since $u=\underline{u}=\varphi$ on $\partial\Omega$, we have
\begin{equation}\label{cal141}
    |u_{t_\alpha t_\beta}(0)|\leq|\varphi_{t_\alpha t_\beta}(0)|+|\rho_{\alpha\beta}(0)(u-\varphi)_{x_n}(0)|\leq C.
\end{equation}

\subsubsection{Tangential-normal derivatives}
Using barrier functions as in \cite{Caffarelli1985-2,li2004dirichlet,guan2023dirichlet}, we can show that in a small neighbourhood $\Omega\cap B_\delta(0)$ of 0,
\begin{equation}\label{cal131}
    L\left(\pm T_\alpha(u-\underline{u})+(u_{y_n}-\underline{u}_{y_n})^2+\sum_{l<n}|u_l-\underline{u}_l|^2\right)\geq-C\left(1+\mathcal{F}\right),
\end{equation}
where $L=F^{i\bar{j}}\partial_{i\bar{j}}$ is the linearization of $F$ at $u$ and $T_\alpha=\partial_{t_\alpha}+\rho_{t_\alpha}\partial_{x_n}$ for any $1\leq\alpha\leq 2n-1$. Indeed,  
\begin{equation}
    L(T_\alpha u)=T_\alpha(S_k^{\frac{1}{k}}(\mu[u]))+F^{i\bar{j}}\rho_{it_\alpha}u_{x_n\bar{j}}+F^{i\bar{j}}\rho_{\bar{j}t_\alpha}u_{x_ni}+F^{i\bar{j}}\rho_{i\bar{j}t_\alpha}u_{x_n}.
\end{equation}The first term vanishes and the last term is bounded by $|F^{i\bar{j}}\rho_{i\bar{j}t_\alpha}u_{x_n}|\leq C\mathcal{F}$. For the two terms in the middle, note that  
\begin{equation}
|F^{i\bar{j}}\rho_{it_\alpha}u_{x_n\bar{j}}|\leq|2F^{i\bar{j}}\rho_{it_\alpha}u_{n\bar{j}}|+|iF^{i\bar{j}}\rho_{it_\alpha}u_{y_n\bar{j}}|
\leq C\sum f_i|\lambda_i|+C\mathcal{F}^{\frac{1}{2}}|F^{i\bar{j}}u_{y_ni}u_{y_n\bar{j}}|^{\frac{1}{2}}, 
\end{equation} 
as a consequence of Cauchy-Schwarz inequality. Combined with $|L(T_\alpha\underline{u})|\leq C\mathcal{F}$, we obtain
\begin{equation}
    |LT_\alpha(u-\underline{u})|\leq C\left(\sum f_i|\lambda_i|+\mathcal{F}^\frac{1}{2}|F^{i\bar{j}}u_{y_ni}u_{y_n\bar{j}}|^{\frac{1}{2}}+\mathcal{F}\right).
\end{equation}
To cancel out the first two terms, note that
\begin{equation}
\begin{aligned}
     L\left(\sum_{l<n}|u_l-\underline{u}_{l}|^2\right)=&
     \sum_{l<n}F^{i\bar{j}}((u_{l\bar{j}}-\underline{u}_{l\bar{j}})(u_{i\bar{l}}-\underline{u}_{i\bar{l}})+(u_{il}-\underline{u}_{il})(u_{\bar{l}\bar{j}}-\underline{u}_{\bar{l}\bar{j}}))\\
     &+\sum_{l<n}F^{i\bar{j}}((u_{li\bar{j}}-\underline{u}_{li\bar{j}})(u_{\bar{l}}-\underline{u}_{\bar{l}})+(u_{\bar{l}i\bar{j}}-\underline{u}_{\bar{l}i\bar{j}})(u_l-\underline{u}_l))\\
     :=&I_1+I_2,
\end{aligned}
\end{equation}where 
\begin{gather}
I_1\geq\sum_{l<n}F^{i\bar{j}}u_{i\bar{l}}u_{l\bar{j}}-2\left(\sum_{l<n}F^{i\bar{j}}u_{i\bar{l}}u_{l\bar{j}}\right)^{\frac{1}{2}}\left(\sum_{l<n}F^{i\bar{j}}\underline{u}_{i\bar{l}}\underline{u}_{l\bar{j}}\right)^{\frac{1}{2}}\geq\frac{1}{2}\sum_{l<n}F^{i\bar{j}}u_{i\bar{l}}u_{l\bar{j}}-C\mathcal{F},\\
    I_2=-\sum_{l<n}F^{i\bar{j}}\underline{u}_{li\bar{j}}(u_{\bar{l}}-\underline{u}_{\bar{l}})+F^{i\bar{j}}\underline{u}_{\bar{l}i\bar{j}}(u_l-\underline{u}_l)\geq-C\mathcal{F}.
\end{gather}
Likewise we derive
\begin{equation}
    L(u_{y_n}-\underline{u}_{y_n})^2\geq F^{i\bar{j}}u_{y_ni}u_{y_nj}-C\mathcal{F}.
\end{equation}
Hence 
\begin{equation}
\begin{aligned}
    L\left((u_{y_n}-\underline{u}_{y_n})^2+\sum_{l<n}|u_l-\underline{u}_l|^2\right)
    \geq\frac{1}{2}\sum_{l<n}F^{i\bar{j}}u_{i\bar{l}}u_{l\bar{j}}+F^{i\bar{j}}u_{y_ni}u_{y_nj}-C\mathcal{F}.
\end{aligned}
\end{equation}
Therefore, we have
\begin{equation}\label{cal132}
\begin{aligned}
    &L\left(\pm T_\alpha(u-\underline{u})+(u_{y_n}-\underline{u}_{y_n})^2+\sum_{l<n}|u_l-\underline{u}_l|^2\right)\\
    \geq&\frac{1}{2}\sum_{l<n}F^{i\bar{j}}u_{i\bar{l}}u_{l\bar{j}}+F^{i\bar{j}}u_{y_ni}u_{y_nj}-C\left(\sum f_i|\lambda_i|+\mathcal{F}^\frac{1}{2}|F^{i\bar{j}}u_{y_ni}u_{y_n\bar{j}}|^{\frac{1}{2}}+\mathcal{F}\right)\\
    \geq&\frac{1}{2}\sum_{l<n}F^{i\bar{j}}u_{i\bar{l}}u_{l\bar{j}}-C\left(\sum f_i|\lambda_i|+\mathcal{F}\right).
\end{aligned}
\end{equation}
Using Proposition 2.6 and Corollary 2.8 in \cite{Guan2014}, for any $\sigma>0$ we have
\begin{equation}
    \sum f_i|\lambda_i|\leq \sigma\sum_{i\neq r} f_i|\lambda_i|^2+\frac{C}{\sigma}\sum f_i+C\leq\frac{\sigma}{c_0}\sum_{l<n}F^{i\bar{j}}u_{i\bar{l}}u_{l\bar{j}}+\frac{C}{\sigma}\mathcal{F}+C
\end{equation}for some index $1\leq r\leq n$ and constant $c_0>0$. Thus choosing $\sigma=\frac{c_0}{2C}$, we obtain \eqref{cal131} from \eqref{cal132}. Set $\psi=u-\underline{u}+td-\frac{N}{2}d^2$ in $\Omega\cap B_\delta(0)$, where $d$ is the distance to $\partial\Omega$ and $t,\delta\ll1,N\gg1$ are fixed constants. By Lemma 6.2 in \cite{Guan1999CVPDE}, it is known that 
\begin{equation}
    \begin{cases}
        L\psi\leq-\varepsilon_0\left(1+\mathcal{F}\right)\quad{\rm{in}}\;\Omega\cap B_\delta(0),\\
        \psi\geq0\quad{\rm{on}}\;\partial(\Omega\cap B_\delta(0)).
    \end{cases}
\end{equation}
Therefore, for $B\gg A\gg1$, we have $Lw\geq0$, where
\begin{equation}
    w=\pm T_\alpha(u-\underline{u})+(u_{y_n}-\underline{u}_{y_n})^2+\sum_{l<n}|u_l-\underline{u}_l|^2-A|z|^2-B\psi.
\end{equation}
A careful examination of $w$ on the boundary implies $w\leq0$ on $\partial(\Omega\cap B_\delta(0))$, provided that we take $A>0$ sufficiently large. Now by the maximum principle, we know $w\leq0$ in $\Omega\cap B_\delta(0)$. It follows that $\tilde{w}:=\pm T_\alpha(u-\underline{u})-A|z|^2-B\psi\leq0$ in $\Omega\cap B_\delta(0)$ with $\tilde{w}(0)=0$. This implies that $\tilde{w}_{x_n}(0)=-\tilde{w}_\nu(0)\leq0$. Therefore, we obtain
\begin{equation}
    |u_{t_\alpha x_n}(0)|=|(T_\alpha u)_{x_n}(0)|\leq|(T_\alpha\underline{u})_{x_n}(0)|+B\psi_{x_n}(0)+\tilde{w}_{x_n}(0)\leq C.
\end{equation}

\subsubsection{Normal derivatives}
To bound $|u_{x_nx_n}(0)|$ we only need to estimate $|u_{n\bar{n}}(0)|$. Since $\lambda(i\partial\bar{\partial}u)\in\Gamma\subset\Gamma_1$, clearly 
\begin{equation}
    u_{n\bar{n}}(0)\geq-\sum_{i<n}u_{i\bar{i}}(0)\geq-C.
\end{equation}
To derive an upper bound we solve equation \eqref{approxdirichletproblem} for $u_{n\bar{n}}(0)$. WLOG $\{u_{l\bar{m}}(0)\}_{1\leq l,m\leq n-1}$ is diagonal and $|u_{l\bar{l}}(0)|,|u_{l\bar{n}}(0)|\leq C$ for all $1\leq l\leq n-1$. Suppose $u_{n\bar{n}}(0)=N$ is very large. Using Lemma 1.2 in \cite{caffarelli1985dirichlet}, we know the eigenvalues of $i\partial\bar{\partial}u$ at 0 are
\begin{equation}
    \lambda_l=u_{l\bar{l}}(0)+o(1),\quad\lambda_n=u_{n\bar{n}}(0)+O(1)=N+O(1),\quad1\leq l\leq n-1.
\end{equation}
Thus 
\begin{equation}
    \mu_l=N+O(1),\quad\mu_n=\sum_{l<n}u_{l\bar{l}}(0)+o(1)=O(1),\quad1\leq l\leq n-1.
\end{equation}
For $k<n$, it follows that
\begin{equation}
    \varepsilon=S_k(\mu[u])=\mu_nS_{k-1}(\mu_1,\cdots,\mu_{n-1})+S_k(\mu_1,\cdots,\mu_{n-1})=N^k+O(N^{k-1}),
\end{equation}
which means that $u_{n\bar{n}}(0)=N\leq C$. For $k=n$, we have instead
\begin{equation}
    \varepsilon=\mu_1\cdots\mu_{n-1}\mu_n=(N^{n-1}+O(N^{n-2}))\mu_n.
\end{equation}
To bound $N$ from above we need a positive lower bound for $\mu_n$, or equivalently, $\sum_{l<n}u_{l\bar{l}}(0)$.
Indeed, since $u=0$ on $\partial\Omega$, we have
\begin{equation}
    u_{t_\alpha t_\beta}(0)=-\rho_{\alpha\beta}(0)u_{x_n}(0).
\end{equation}
A local defining function of $\partial\Omega$ near 0 is given by
\begin{equation}
    \sigma(z)=-x_n+{\rm Re}(\sigma_{ij}(0)z_iz_j)+\sigma_{i\bar{j}}(0)z_i\bar{z}_j+O(|z|^3).
\end{equation}
Locally we have $\sigma(t',\rho(t'))=0$. Hence for $1\leq l,m\leq n-1$,
\begin{equation}
    \sigma_{l\bar{m}}(0)=-\rho_{l\bar{m}}(0)\sigma_{x_n}(0)=\frac{\sigma_{x_n}(0)}{u_{x_n}(0)}u_{l\bar{m}}(0)=-\frac{1}{u_{x_n}(0)}u_{l\bar{m}}(0).
\end{equation}
Now by \eqref{cal006'}, we have $-u_{x_n}(0)=u_\nu(0)\geq c_1>0$. Since $\Omega$ is assumed to be pseudo mean-convex as in Definition \ref{pseudoconvexity}, we assert that
\begin{equation}
    \sum_{l<n}u_{l\bar{l}}(0)={\rm tr}u_{l\bar{m}}(0)=-u_{x_n}(0){\rm tr}\sigma_{l\bar{m}}(0)\geq c_0c_1>0.
\end{equation}

\subsection{Global second order estimates, $k<n$}

We follow the idea in \cite{Chou2001AVT} to prove \eqref{C2estimate} using \eqref{C1estimate} and \eqref{C2boundary1}. For brevity, we write $u$ instead of $u^\varepsilon$ and $u<0$ throughout the proof. Define $P$ as in \eqref{generalP}, but of the more specific form
    \begin{equation}
        P(z)=(-u)^{-a}(|z|^2-\varepsilon^2)^b|\nabla u|^2,
    \end{equation}
    where $a,b\in\mathbb{R}$ are nonnegative constants to be chosen. In both cases we will need the following lemma:
    \begin{lemma}\label{lemma001}
        For any $z\in\Omega_\varepsilon$, we have
        \begin{equation}
            F^{i\bar{j}}\partial_{i\bar{j}}P(z)\geq P\left(\frac{1}{2|\nabla u|^2}F^{i\bar{j}}u_{i\bar{l}}u_{l\bar{j}}-\frac{a(4a-1)}{u^2}F^{i\bar{j}}u_iu_{\bar{j}}-\frac{b(4b+1)}{(|z|^2-\varepsilon^2)^2}F^{i\bar{j}}\bar{z}_iz_j\right).
        \end{equation}
    \end{lemma}
    \begin{proof}
        After some straightforward calculations, we find
        \begin{equation}\label{cal111}
        \begin{aligned}
            P^{-1}F^{i\bar{j}}P_{i\bar{j}}=&\frac{F^{i\bar{j}}\partial_{i\bar{j}}|\nabla u|^2}{|\nabla u|^2}-\frac{F^{i\bar{j}}\partial_i|\nabla u|^2\partial_{\bar{j}}|\nabla u|^2}{|\nabla u|^4}\\
            &-a\frac{F^{i\bar{j}}u_{i\bar{j}}}{u}+a\frac{F^{i\bar{j}}u_iu_{\bar{j}}}{u^2}+b\frac{F^{i\bar{j}}\delta_{ij}}{|z|^2-\varepsilon^2}-b\frac{F^{i\bar{j}}\bar{z}_iz_j}{(|z|^2-\varepsilon^2)^2}\\
            &+F^{i\bar{j}}\left(-a\frac{u_i}{u}+b\frac{\bar{z}_i}{|z|^2-\varepsilon^2}+\frac{\partial_i|\nabla u|^2}{|\nabla u|^2}\right)\left(-a\frac{u_{\bar{j}}}{u}+b\frac{z_j}{|z|^2-\varepsilon^2}+\frac{\partial_{\bar{j}}|\nabla u|^2}{|\nabla u|^2}\right)\\
            =&\frac{F^{i\bar{j}}(u_{il}u_{\bar{l}\bar{j}}+u_{i\bar{l}}u_{l\bar{j}})}{|\nabla u|^2}-a\frac{F^{i\bar{j}}u_{i\bar{j}}}{u}+a\frac{F^{i\bar{j}}u_iu_{\bar{j}}}{u^2}+b\frac{\mathcal{F}}{|z|^2-\varepsilon^2}-b\frac{F^{i\bar{j}}\bar{z}_iz_j}{(|z|^2-\varepsilon^2)^2}\\
            &+F^{i\bar{j}}\left(-a\frac{u_i}{u}+b\frac{\bar{z}_i}{|z|^2-\varepsilon^2}\right)\left(-a\frac{u_{\bar{j}}}{u}+b\frac{z_j}{|z|^2-\varepsilon^2}\right)\\
            &+2{\rm Re}\left(F^{i\bar{j}}\left(\frac{u_{il}u_{\bar{l}}+u_{i\bar{l}}u_l}{|\nabla u|^2}\right)\left(-a\frac{u_{\bar{j}}}{u}+b\frac{z_j}{|z|^2-\varepsilon^2}\right)\right).
        \end{aligned}
        \end{equation}
        Using Cauchy-Schwarz inequality, we have
        \begin{equation}
        \begin{aligned}
             &\left|2{\rm Re}\left(F^{i\bar{j}}u_{il}u_{\bar{l}}\left(-a\frac{u_{\bar{j}}}{u}+b\frac{z_j}{|z|^2-\varepsilon^2}\right)\right)\right|\\
             \leq&|\nabla u|^2F^{i\bar{j}}\left(-a\frac{u_i}{u}+b\frac{\bar{z}_i}{|z|^2-\varepsilon^2}\right)\left(-a\frac{u_{\bar{j}}}{u}+b\frac{z_j}{|z|^2-\varepsilon^2}\right)+F^{i\bar{j}}u_{il}u_{\bar{l}\bar{j}}
        \end{aligned}
        \end{equation}
        and 
        \begin{equation}
        \begin{aligned}
             &\left|2{\rm Re}\left(F^{i\bar{j}}u_{i\bar{l}}u_l\left(-a\frac{u_{\bar{j}}}{u}+b\frac{z_j}{|z|^2-\varepsilon^2}\right)\right)\right|\\
             \leq&2|\nabla u|^2F^{i\bar{j}}\left(-a\frac{u_i}{u}+b\frac{\bar{z}_i}{|z|^2-\varepsilon^2}\right)\left(-a\frac{u_{\bar{j}}}{u}+b\frac{z_j}{|z|^2-\varepsilon^2}\right)+\frac{1}{2}F^{i\bar{j}}u_{i\bar{l}}u_{l\bar{j}}.
        \end{aligned}
        \end{equation}
        Substituting into \eqref{cal111} and using Cauchy-Schwarz inequality again, we obtain
        \begin{equation}
        \begin{aligned}
            P^{-1}F^{i\bar{j}}P_{i\bar{j}}\geq&-a\frac{F^{i\bar{j}}u_{i\bar{j}}}{u}+a\frac{F^{i\bar{j}}u_iu_{\bar{j}}}{u^2}+b\frac{\mathcal{F}}{|z|^2-\varepsilon^2}-b\frac{F^{i\bar{j}}\bar{z}_iz_j}{(|z|^2-\varepsilon^2)^2}\\
            &-2F^{i\bar{j}}\left(-a\frac{u_i}{u}+b\frac{\bar{z}_i}{|z|^2-\varepsilon^2}\right)\left(-a\frac{u_{\bar{j}}}{u}+b\frac{z_j}{|z|^2-\varepsilon^2}\right)\\
            \geq&\frac{{F^{i\bar{j}}u_{i\bar{l}}u_{l\bar{j}}}}{2|\nabla u|^2}-(4a^2-a)\frac{F^{i\bar{j}}u_iu_{\bar{j}}}{u^2}-(b+4b^2)\frac{F^{i\bar{j}}\bar{z}_iz_j}{(|z|^2-\varepsilon^2)^2},
        \end{aligned}
        \end{equation}
        where we use the fact that $F^{i\bar{j}}u_{i\bar{j}}>0,u<0,|z|>\varepsilon$ in $\Omega_\varepsilon$.
    \end{proof}
    Now back to the proof of \eqref{C2estimate}. We know from \eqref{C1estimate} that when $k<n,a=2-\frac{2}{\gamma},b=2$, $P$ is uniformly bounded in $\overline{\Omega}_\varepsilon$. Denote $M=2\max_{\overline{\Omega}_\varepsilon}P>0$, which is independent of $\varepsilon$. Consider the function
    \begin{equation}
        Q(z,\xi)=e^{\varphi(P)+\phi(u)+\psi(|z|^2-\varepsilon^2)}u_{\xi\bar{\xi}},\quad z\in\overline{\Omega}_\varepsilon,\xi\in\mathbb{C}^n,
    \end{equation}where $\varphi(P)=-\tau\log(1-\frac{P}{M})$ satisfies $\varphi''=\frac{1}{\tau}|\varphi'|^2$. Here $\tau>0$ is a small fixed constant to be determined. It follows that $0\leq\varphi(P)\leq\tau\log2$ and that $\varphi'(P)=\tau(M-P)^{-1},\frac{\tau}{M}\leq\varphi'(P)\leq\frac{2\tau}{M}$. 
    
    Suppose $G$ attains a maximum at some $z_0\in\Omega_\varepsilon$ in the direction $\xi=(1,0,\cdots,0)$. WLOG we assume $u_{i\bar{j}}(z_0)$ is diagonal so that $\lambda_i=u_{i\bar{i}}$, and that $\lambda_1\geq\cdots\geq\lambda_n,F^{n\bar{n}}\geq\cdots\geq F^{1\bar{1}}>0$. The derivative tests then imply at $z_0$,
    \begin{equation}\label{gradienttest}
        0=\partial_i\log Q=\varphi'P_i+\phi'u_i+\psi'\bar{z}_i+\frac{u_{1\bar{1}i}}{u_{1\bar{1}}},\quad 0=\partial_{\bar{i}}\log Q=\varphi'P_{\bar{i}}+\phi'u_{\bar{i}}+\psi'z_i+\frac{u_{1\bar{1}\bar{i}}}{u_{1\bar{1}}},
    \end{equation}
and 
\begin{equation}\label{2ndderivativetest}
\begin{aligned}
     0&\geq F^{i\bar{i}}\partial_{i\bar{i}}\log Q\\    &=\varphi'F^{i\bar{i}}P_{i\bar{i}}+\varphi''F^{i\bar{i}}|P_i|^2+\phi'F^{i\bar{i}}u_{i\bar{i}}+\phi''F^{i\bar{i}}|u_i|^2+\psi'\mathcal{F}+\psi''F^{i\bar{i}}|z_i|^2+F^{i\bar{i}}\left(\frac{u_{1\bar{1}i\bar{i}}}{u_{1\bar{1}}}-\frac{|u_{1\bar{1}i}|^2}{u_{1\bar{1}}^2}\right).
\end{aligned}
\end{equation}

By \eqref{gradienttest} and Cauchy-Schwarz inequality,
\begin{equation}\label{cal112}
    \frac{|u_{1\bar{1}i}|^2}{u_{1\bar{1}}^2}=|\varphi'P_i+\phi'u_i+\psi'\bar{z}_i|^2\leq3(|\varphi'|^2|P_i|^2+|\phi'|^2|u_i|^2+|\psi'|^2|z_i|^2).
\end{equation}
Since $F$ is concave, we have, using Lemma \ref{lemma001} and \eqref{cal112} in \eqref{2ndderivativetest},
\begin{equation}
\begin{aligned}
    0\geq&\frac{\tau P}{M}\frac{1}{2|\nabla u|^2}F^{i\bar{i}}u_{i\bar{i}}^2-\frac{2\tau P}{M}\frac{a(4a-1)}{u^2}F^{i\bar{i}}|u_i|^2-\frac{2\tau P}{M}\frac{18}{(|z|^2-\varepsilon^2)^2}F^{i\bar{i}}|z_i|^2\\
    &+(\varphi''-3|\varphi'|^2)F^{i\bar{i}}|P_i|^2+(\phi''-3|\phi'|^2)F^{i\bar{i}}|u_i|^2+(\psi''-3|\psi'|^2)F^{i\bar{i}}|z_i|^2+\phi'\varepsilon^{\frac{1}{k}}+\psi'\mathcal{F},
\end{aligned}
\end{equation}
where $a=2-\frac{2}{\gamma}$. Let $\tau=\frac{1}{3}$ and $\phi(u)=-(a-1)\log(-u),\psi(|z|^2-\varepsilon^2)=2\log(|z|^2-\varepsilon^2)$. It follows that 
\begin{equation}
    0\geq\frac{1}{6M}(-u)^{-a}(|z|^2-\varepsilon^2)^2F^{i\bar{i}}u_{i\bar{i}}^2-\frac{13a^2-22a+12}{3}\frac{F^{i\bar{i}}|u_i|^2}{u^2}-16\frac{F^{i\bar{i}}|z_i|^2}{(|z|^2-\varepsilon^2)^2}.
\end{equation}
By \eqref{uniformlyelliptic}, we have $F^{1\bar{1}}\geq c_0\mathcal{F}$, which implies that
\begin{equation}
    0\geq\frac{c_0}{M}(-u)^{-a}(|z|^2-\varepsilon^2)^2u_{1\bar{1}}^2-C_1\frac{|\nabla u|^2}{u^2}-C_2\frac{|z|^2}{(|z|^2-\varepsilon^2)^2}.
\end{equation}
Multiplying both sides by $M(1-\frac{P}{M})^{-2\tau}(-u)^{2-a}(|z|^2-\varepsilon^2)^2$ yields
\begin{equation}
    0\geq c_0Q^2-C_1MP-C_2
\end{equation}
where we use \eqref{C0estimate}. This proves $Q\leq C_{n,k}M$ is bounded in $\Omega_\varepsilon$ independent of $\varepsilon$. On the other hand, $Q(z)=0$ for $|z|=\varepsilon$ and by \eqref{C2boundary1}, $Q$ is also bounded on $\partial\Omega$. Combining \eqref{C0estimate} we deduce \eqref{C2estimate}. 

\subsection{Global second order estimates, $k=n$}

We integrate the idea in \cite{szekelyhidi2018fully,gao2023exterior}. Still writing $u$ instead of $u^\varepsilon$, we now define $P$ as in \eqref{P'}, namely
\begin{equation}
    P=(-u)^{-\beta}|\nabla u|^2,\quad \beta=2+\frac{1}{n-2},\quad M=2\max_{\overline{\Omega}_\varepsilon}P.
\end{equation}
By \eqref{C1estimate'}, we know $M$ is finite and independent of $\varepsilon$. Consider the function
    \begin{equation}
        Q(z,\xi)=e^{\varphi(P)+\phi(u)}u_{\xi\bar{\xi}},\quad z\in\overline{\Omega}_\varepsilon,\xi\in\mathbb{C}^n,
    \end{equation}where $\varphi(P)=-\tau\log(1-\frac{P}{M}),\tau>0$ is a small fixed constant to be determined, and $\phi(u)=-\alpha\log(-u)$ for $\alpha=\beta-1>1$. 

     Suppose $Q$ attains an interior maximum at some $z_0\in\Omega_\varepsilon$ in the direction $\xi=(1,0,\cdots,0)$. WLOG we assume $u_{i\bar{j}}(z_0)$ is diagonal so that $\lambda_i=u_{i\bar{i}}$, and that $\lambda_1\geq\cdots\geq\lambda_n,F^{n\bar{n}}\geq\cdots\geq F^{1\bar{1}}>0$. The derivative tests then imply at $z_0$,
        \begin{equation}\label{gradienttest'}
        0=\partial_i\log Q=\varphi'P_i+\phi'u_i+\frac{u_{1\bar{1}i}}{u_{1\bar{1}}},\quad 0=\partial_{\bar{i}}\log Q=\varphi'P_{\bar{i}}+\phi'u_{\bar{i}}+\frac{u_{1\bar{1}\bar{i}}}{u_{1\bar{1}}},
    \end{equation}
and 
\begin{equation}\label{2ndderivativetest'}
    0\geq F^{i\bar{i}}\partial_{i\bar{i}}\log Q  =\varphi'F^{i\bar{i}}P_{i\bar{i}}+\varphi''F^{i\bar{i}}|P_i|^2+\phi'F^{i\bar{i}}u_{i\bar{i}}+\phi''F^{i\bar{i}}|u_i|^2+F^{i\bar{i}}\left(\frac{u_{1\bar{1}i\bar{i}}}{u_{1\bar{1}}}-\frac{|u_{1\bar{1}i}|^2}{u_{1\bar{1}}^2}\right).
\end{equation}

Fix some small constant $0<\delta<\frac{1}{4\alpha}<\frac{1}{4}$. Consider the following two cases.
    \begin{enumerate}
    
        \item $\lambda_n\leq-\delta\lambda_1$.
        
In this scenario, the proof is similar to the $k<n$ case. Using \eqref{gradienttest'} to cancel out the third derivative terms, we get, after using Lemma \ref{lemma001} for $a=\beta,b=0$ and noting that $F$ is concave,

\begin{equation}
\begin{aligned}
    0&\geq\frac{\tau P}{M}\frac{1}{2|\nabla u|^2}F^{i\bar{i}}u_{i\bar{i}}^2+(\varphi''-2|\varphi'|^2)F^{i\bar{i}}|P_i|^2+\left(\phi''-2|\phi'|^2-\frac{2\tau P}{M}\frac{\beta(4\beta-1)}{u^2}\right)F^{i\bar{i}}|u_i|^2\\
    &\geq\frac{\tau}{2M}(-u)^{-\beta}F^{i\bar{i}}\lambda_i^2-(\alpha(2\alpha-1)+\tau\beta(4\beta-1))\frac{F^{i\bar{i}}|u_i|^2}{u^2}\\
    &\geq\frac{\tau\delta^2}{2M}(-u)^{-\beta}F^{n\bar{n}}\lambda_1^2-C_{n,\tau}\mathcal{F}\frac{|\nabla u|^2}{u^2},
\end{aligned}
\end{equation}
where we take $\tau<\frac{1}{2}$ sufficiently small. Clearly we have $F^{n\bar{n}}\geq \frac{1}{n}\mathcal{F}$, so multiplying both sides by $M(1-\frac{P}{M})^{-2\tau}(-u)^{2-\beta}$ yields
\begin{equation}
    0\geq\frac{\tau\delta^2}{2n}Q^2-C_{n,\tau}MP,
\end{equation}
namely $Q\leq C_{n,\tau,\delta}M$.

        \item $\lambda_n>-\delta\lambda_1$.

Note that
\begin{equation}
    \phi'(u)=-\frac{\alpha}{u},\quad\phi''(u)=\frac{\alpha}{u^2}=\frac{1}{\alpha}|\phi'(u)|^2.
\end{equation}
Using Lemma \ref{lemma001} again for $a=\beta,b=0$, we have
\begin{equation}
    \varphi'F^{i\bar{i}}P_{i\bar{i}}\geq\frac{\tau}{2M}(-u)^{-\beta}F^{i\bar{i}}\lambda_i^2-\tau\beta(4\beta-1)F^{i\bar{i}}\frac{|u_i|^2}{u^2}.
\end{equation}
Now take $0<\tau\leq\frac{\alpha-4\delta\alpha^2}{\beta(4\beta-1)}<\frac{1}{2}$ and rewrite \eqref{2ndderivativetest'} as
\begin{equation}
    0\geq\frac{\tau}{2M}(-u)^{-\beta}F^{i\bar{i}}\lambda_i^2+\varphi''F^{i\bar{i}}|P_i|^2+4\delta\alpha^2F^{i\bar{i}}\frac{|u_i|^2}{u^2}+F^{i\bar{i}}\frac{u_{1\bar{1}i\bar{i}}}{\lambda_1}-F^{i\bar{i}}\frac{|u_{1\bar{1}i}|^2}{\lambda_1^2}.
\end{equation}

Let $I=\{1\leq i\leq n:F^{1\bar{1}}<\delta F^{i\bar{i}}\}$. For $i\in I$, we substitute $F^{i\bar{i}}|u_i|^2$ using \eqref{gradienttest'} and the absorbing inequality:
\begin{equation}\label{cal121}
    \alpha^2\frac{|u_i|^2}{u^2}=|\phi'|^2|u_i|^2=\left|\frac{u_{1\bar{1}i}}{\lambda_1}+\varphi'P_i\right|^2\geq\frac{1}{2}\frac{|u_{1\bar{1}i}|^2}{\lambda_1^2}-|\varphi'|^2|P_i|^2.
\end{equation}
For $i\notin I$, we get rid of the third order terms as before:
\begin{equation}\label{cal122}
    \frac{|u_{1\bar{1}i}|^2}{\lambda_1^2}=|\varphi'P_i+\phi'u_i|^2\leq2|\varphi'|^2|P_i|^2+2\alpha^2\frac{|u_i|^2}{u^2}.
\end{equation}
Combining \eqref{cal121}, \eqref{cal122} with \eqref{2ndderivativetest'}, we have
\begin{equation}\label{cal123}
\begin{aligned}
        0\geq&\frac{\tau}{2M}(-u)^{-\beta}F^{i\bar{i}}\lambda_i^2+(\varphi''-2|\varphi'|^2)\sum_{i\notin I}F^{i\bar{i}}|P_i|^2+(\varphi''-4\delta|\varphi'|^2)\sum_{i\in I}F^{i\bar{i}}|P_i|^2\\
        &-(2-4\delta)\alpha^2\sum_{i\notin I}F^{i\bar{i}}\frac{|u_i|^2}{u^2}+F^{i\bar{i}}\frac{u_{1\bar{1}i\bar{i}}}{\lambda_1}-(1-2\delta)\sum_{i\in I}F^{i\bar{i}}\frac{|u_{1\bar{1}i}|^2}{\lambda_1^2}.
\end{aligned}
\end{equation}
Since $\frac{1}{\tau}>2>4\delta$ and $\varphi''=\frac{1}{\tau}|\varphi'|^2$, the second and the third term in \eqref{cal123} are positive. If we can show the sum of the last two terms is also positive, then since $F^{i\bar{i}}\leq\delta^{-1}F^{1\bar{1}}$ for all $i\notin I$, we have
\begin{equation}
\begin{aligned}
        0&\geq\frac{\tau}{2M}(-u)^{-\beta}F^{i\bar{i}}\lambda_i^2-(2-4\delta)\alpha^2\sum_{i\notin I}F^{i\bar{i}}\frac{|u_i|^2}{u^2}\\
        &\geq\frac{\tau}{2M}(-u)^{-\beta}F^{1\bar{1}}\lambda_1^2-(\frac{2}{\delta}-4)\alpha^2F^{1\bar{1}}\frac{|\nabla u|^2}{u^2}.
\end{aligned}
\end{equation}
Multiplying both sides by $M(1-\frac{P}{M})^{-2\tau}(-u)^{2-\beta}$, we get the same
\begin{equation}
    0\geq\frac{\tau}{2}Q^2-C_{n,\tau,\delta}MP.
\end{equation}

In fact, by \eqref{linearizedderivative2} and \eqref{2ndderivativeofF}, taking derivatives in the directions $\partial_1,\partial_{\bar{1}}$ yields
\begin{equation}
\begin{aligned}
        F^{i\bar{i}}{u_{1\bar{1}i\bar{i}}}=&-F^{i\bar{j},l\bar{m}}u_{i\bar{j}1}u_{l\bar{m}\bar{1}}
        =-\sum_{i,j}f_{ij}u_{i\bar{i}1}u_{j\bar{j}1}+\sum_{p\neq q}\frac{F^{p\bar{p}}-F^{q\bar{q}}}{\lambda_q-\lambda_p}|u_{p\bar{q}1}|^2\\
        &\geq\sum_{i\in I}\frac{F^{i\bar{i}}-F^{1\bar{1}}}{\lambda_1-\lambda_i}|u_{1\bar{1}i}|^2>\sum_{i\in I}\frac{(1-\delta)F^{i\bar{i}}}{\lambda_1-\lambda_i}|u_{1\bar{1}i}|^2,
\end{aligned}
\end{equation}
where we use the concavity of $f$. Thus
\begin{equation}
    F^{i\bar{i}}\frac{u_{1\bar{1}i\bar{i}}}{\lambda_1}-(1-2\delta)\sum_{i\in I}F^{i\bar{i}}\frac{|u_{1\bar{1}i}|^2}{\lambda_1^2}>\frac{1}{\lambda_1}\sum_{i\in I}\left(\frac{1-\delta}{\lambda_1-\lambda_i}-\frac{(1-2\delta)}{\lambda_1}\right)F^{i\bar{i}}|u_{1\bar{1}i}|^2.
\end{equation}
To show this is positive, we only need 
\begin{equation}
    (1-\delta)\lambda_1-(1-2\delta)(\lambda_1-\lambda_i)=\delta\lambda_1+(1-2\delta)\lambda_i\geq0.
\end{equation}
Indeed, if $\lambda_i\geq0$, then there is nothing to prove. If $\lambda_i<0$, since $\lambda_i>\lambda_n>-\delta\lambda_1$, we know $\delta\lambda_1+(1-2\delta)\lambda_i>-2\delta\lambda_i>0$. 

\end{enumerate}

This concludes the global complex Hessian estimate and reduces the second order estimates to the boundary case. Considering \eqref{C2boundary1}, it suffices to show $|\nabla^2u|\leq C\varepsilon^{-\gamma-2}$ on $\partial B_\varepsilon$, or equivalently, $|\nabla^2\tilde{u}(z)|\leq C$ for $z\in\partial B_1$ after scaling $u^\varepsilon$ and $\underline{u}$ to $\tilde{u}$ and $\tilde{\underline{u}}$. By \eqref{C1estimate'} and \eqref{estimatesforsubsolution}, we have
\begin{equation}
    \|\tilde{u}\|_{C^1(B_2\setminus B_1)}\leq C,\quad \|\tilde{\underline{u}}\|_{C^4(\overline{D}\setminus B_1)}\leq C
\end{equation}for some constant $C$ uniform in $\varepsilon$. WLOG $z_0=(0,\cdots,0,1)\in\partial B_1$. Write $z=(t',x_n)\in\mathbb{R}^{2n}$, where $t'=(t_1,\cdots,t_{2n-1}), t_{2i-1}=y_i,t_{2i}=x_i,1\leq i\leq n-1,t_{2n-1}=y_n$. Locally near $z_0$, $\partial B_1$ is represented by
\begin{equation}
    x_n=\rho(t')=(1-|t'|^2)^{\frac{1}{2}}=1-\delta_{\alpha\beta}t_\alpha t_\beta+O(|t'|^3),
\end{equation}where $\alpha,\beta$ sum over $1,\cdots,2n-1$. The proof now proceeds similarly as the $\partial\Omega$ case, only with a `concave' and also more specific boundary. 

The tangential derivatives are estimated by taking derivatives of $(\tilde{u}-\tilde{\underline{u}})(t',\rho(t'))=0$ in $t_\alpha,t_\beta$ on $\partial B_1$, which implies $|\tilde{u}_{t_\alpha t_\beta}(z_0)|\leq C$ as in \eqref{cal141}. 

As for the tangential-normal derivatives, using similar barrier functions, we still have $Lw\leq 0$ in $(D\setminus B_1)\cap B_\delta(z_0)$, where 
\begin{gather}
    w=\pm T_\alpha(\tilde{u}-\tilde{\underline{u}})+(\tilde{u}_{y_n}-\tilde{\underline{u}}_{y_n})^2+\sum_{l<n}|\tilde{u}_l-\tilde{\underline{u}}_l|^2-A|x_n-1|-B\psi,\\
    \psi=\tilde{u}-\tilde{\underline{u}}+td-\frac{N}{2}d^2\geq0,\quad d(z)={\rm dist}(z,\partial B_1)
\end{gather}
are defined in $(D\setminus B_1)\cap B_\delta(z_0)$ for $t,\delta\ll1,N\gg1$ and $B\gg A\gg1$ sufficiently large. To use the maximum principle, we only need to justify $w\leq0$ on $\partial((D\setminus B_1)\cap B_\delta(z_0))$, which consists of two parts. On $\partial B_1\cap B_\delta(z_0)$, we have $T_\alpha(\tilde{u}-\tilde{\underline{u}})=\psi=0$ and $|x_n-1|=\frac{|t'|^2}{1+x_n}\geq\frac{1}{2}|t'|^2$. It follows that 
\begin{equation}
\begin{aligned}
        (\tilde{u}_{y_n}-\tilde{\underline{u}}_{y_n})^2+\sum_{l<n}|\tilde{u}_l-\tilde{\underline{u}}_l|^2&=\rho_{y_n}^2(\tilde{u}_{x_n}-\tilde{\underline{u}}_{x_n})^2+\frac{1}{4}\sum_{l<n}(\rho_{x_l}^2+\rho_{y_l}^2)(\tilde{u}_{x_n}-\tilde{\underline{u}}_{x_n})^2\\
        &\leq C\frac{y_n^2}{\rho^2}+C\sum_{l<n}\frac{x_l^2+y_l^2}{\rho^2}\leq C\frac{|t'|^2}{1-|t'|^2},
\end{aligned}
\end{equation}
so that 
\begin{equation}
    w\leq C\frac{|t'|^2}{1-|t'|^2}-\frac{A}{2}|t'|^2\leq \left(\frac{C}{1-\delta^2}-\frac{A}{2}\right)|t'|^2\leq0
\end{equation}
for $A>0$ sufficiently large. On $\partial B_\delta(z_0)\cap B_1^c$, note that $|x_n-1|+\psi$ is always a positive function on a compact set, and hence assumes a positive lower bound. Taking $A,B$ large further and observing that the first three terms of $w$ are bounded, we assert that $w\leq0$ on $\partial(B_1^c\cap B_\delta(z_0))$.
The maximum principle then implies $w\leq 0$ in $(D\setminus B_1)\cap B_\delta(z_0)$. It follows that $|T_\alpha(\tilde{u}-\tilde{\underline{u}})|\leq A|x_n-1|+B\psi$. In particular, at $t'=0,x_n\rightarrow1$, we have
\begin{equation}
    |\tilde{u}_{t_\alpha x_n}(z_0)|=|(T_\alpha\tilde{u})_{x_n}(z_0)|\leq|(T_\alpha\tilde{\underline{u}})_{x_n}(z_0)|+A+B|\psi_{x_n}(z_0)|\leq C.
\end{equation}

Eventually we need an upper bound for the normal derivatives $u_{n\bar{n}}(z_0)$.
Taking tangential derivatives, since $\tilde{u}_\nu\leq\tilde{\underline{u}}_\nu$ on $\partial B_1$, we have, as a matrix,
\begin{equation}
\begin{aligned}
\tilde{u}_{t_\alpha t_\beta}(z_0)&=\tilde{\underline{u}}_{t_\alpha t_\beta}(z_0)-\delta_{\alpha\beta}(\tilde{u}-\tilde{\underline{u}})_\nu(z_0)\geq\tilde{\underline{u}}_{t_\alpha t_\beta}(z_0)\\
&\geq-\varepsilon^{\gamma+2}\|v\|_{C^2}-\frac{\partial^2|z|^{-\gamma}}{\partial t_\alpha\partial t_\beta}(z_0)
 \geq-C\varepsilon^{\gamma+2}+\gamma\delta_{\alpha\beta}\geq c_0'\delta_{\alpha\beta},
\end{aligned}
\end{equation}
which implies that the matrix
\begin{equation}
    \tilde{u}_{l\bar{m}}(z_0)=\frac{1}{4}(\tilde{u}_{x_l x_m}+\tilde{u}_{y_l y_m}+i(\tilde{u}_{x_l y_m}-\tilde{u}_{x_m y_l}))\geq c_0\delta_{lm}
\end{equation}
is strictly positive definite for for $1\leq l,m\leq n-1$. As in the $\partial\Omega$ case, this guarantees that the sum $\mu_n$ of the $n-1$ tangential eigenvalues of $i\partial\bar{\partial}u$ has a uniform positive lower bound, while the rest of the $\mu_i$'s are approximately as large as $u_{n\bar{n}}(z_0)$. Using equation \eqref{approxdirichletproblem} we know the product $\mu_1\cdots\mu_n=\varepsilon$ is fixed, which implies that $u_{n\bar{n}}(z_0)$ must be bounded from above. 

\section*{Acknowledgements}
The author would like to thank his advisor Prof. Congming Li as well as Prof. Bo Guan for many enlightening discussions and encouragements.
\bibliographystyle{siam}
\bibliography{ref}
\end{document}